\documentclass[12pt]{amsart}

\usepackage{amscd,latexsym,amsthm,amsfonts,amssymb,amsmath,amsxtra,color}
\usepackage[mathscr]{eucal}
\usepackage[pdfstartview=FitH]{hyperref}
\usepackage{pdfsync}
\usepackage[backgroundcolor=green!40, linecolor=green!40]{todonotes}
\usepackage{color}
\usepackage{adjustbox}
\renewcommand\theequation{\thesection.\arabic{equation}}

\newcommand{\BC}{\mathbb {C}}

\newcommand{\BF}{\mathbb {F}}

\newcommand{\BZ}{\mathbb {Z}}

\makeatletter
\def\Ddots{\mathinner{\mkern1mu\raise\p@
\vbox{\kern7\p@\hbox{.}}\mkern2mu
\raise4\p@\hbox{.}\mkern2mu\raise7\p@\hbox{.}\mkern1mu}}
\makeatother

\newcommand{\CJ}{\mathcal {J}}

\newcommand{\CO}{\mathcal {O}}
\newcommand{\CP}{\mathcal {P}}

\newcommand{\CS}{\mathcal {S}}

\newcommand{\scusp}{\mathrm{sc}}

\newcommand{\Ad}{\operatorname{Ad}}
\newcommand{\Aut}{\operatorname{Aut}}

\def\bsN{{\boldsymbol{N}}}

\newcommand{\Gal}{\operatorname{Gal}}
\newcommand{\GL}{\operatorname{GL}}

\newcommand{\Ind}{\operatorname{Ind}}

\newcommand{\Irr}{\operatorname{Irr}}
\newcommand{\Irrgen}{\Irr^{\operatorname{gen}}}

\newcommand{\oIrrtemp}{\overline{\Irr}^{\operatorname{temp}}}
\newcommand{\oIrrchigen}{\overline{\Irr}^{\xi\operatorname{-gen}}}
\newcommand{\oPhitemp}{\overline{\Phi}^{\operatorname{temp}}}
\newcommand{\oPhigen}{\overline{\Phi}^{\operatorname{gen}}}

\newcommand{\Id}{\operatorname{Id}}

\renewcommand{\mid}{\,:\,}

\newcommand{\N}{\operatorname{N}}

\newcommand{\PGL}{\operatorname{PGL}}
\newcommand{\PGSp}{\operatorname{PGSp}}

\newcommand{\Res}{\operatorname{Res}}

\newcommand{\SL}{\operatorname{SL}}

\newcommand{\OO}{\operatorname{O}}
\newcommand{\SO}{\operatorname{SO}}
\newcommand{\G}{\mathrm{G}}

\newcommand{\Sp}{\operatorname{Sp}}
\newcommand{\st}{\operatorname{st}}

\newcommand{\tr}{\operatorname{tr}}

\newcommand{\val}{\operatorname{val}}
\newcommand{\Wedge}{\mathord{\adjustbox{valign=B,totalheight=.6\baselineskip}{$\bigwedge$}}}
\newcommand{\WD}{{WD}}

\newcommand{\ts}{\textstyle}

\def\boI{\boldsymbol 1}

\def\diag{\operatorname{diag}}

\def\L{\mathrm{L}}

\def\rhostd{\varrho_{\mathrm{std}}} 
\def\dss{\mathrm{ds}}

\def\vG{{}^\vee\G}
\def\vL{{}^\vee\L}

\newtheorem{thm}[equation]{Theorem}
\newtheorem{cor}[equation]{Corollary}
\newtheorem{lem}[equation]{Lemma}
\newtheorem{prop}[equation]{Proposition}
\newtheorem{conj}[equation]{Conjecture}

\newtheorem{ques/conj}[equation]{Question/Conjecture}
\newtheorem{ques}[equation]{Question}
\newtheorem{defn}[equation]{Definition}
\newtheorem{rmk}[equation]{Remark}

\newtheorem{exmp}[equation]{Example}

\def\lcttheoremtitle{Local Converse Theorem}
\newtheorem*{lcttheorem}{The~\lcttheoremtitle}
\def\classicallcttheoremtitle{Standard Local Converse Theorem for $G_N$}
\newtheorem*{classicallcttheorem}{The~\classicallcttheoremtitle}

\makeatletter

\newcommand{\Rmnum}[1]{\expandafter\@slowromancap\romannumeral #1@}
\makeatother

\def\ignore#1{\relax}

\def\shaun#1{\textcolor{black}{#1}}
\def\todo#1{\textcolor{red}{#1}}
\definecolor{dgreen}{RGB}{0,205,10}

\def\bJ{\boldsymbol{\CJ}}
\def\bN{\boldsymbol{N}}
\def\br{\boldsymbol{r}}

\begin{document}
\renewcommand{\theequation}{\arabic{equation}}
\numberwithin{equation}{section}

\title[Local Converse Problem]{On sharpness in Local Converse Theorems for classical groups and $G_2$}

\author{Moshe Adrian}
\address{Department of Mathematics\\
Queens College, CUNY\\
Flushing, NY 11367}
\email{moshe.adrian@qc.cuny.edu}

\author{Shaun Stevens}
\address{
University of East Anglia, UK}
\email{shaun.stevens@uea.ac.uk}

\begin{abstract}
We prove various results about the Local Converse Problem for split reductive groups~$\G$ over a non-archimedean local field~$F$ of characteristic~$0$ and residual characteristic~$p$. In particular, we prove that when~$\G$ is a symplectic or special orthogonal group, or the exceptional group~$\G_2$, and~$p$ is large enough, then the optimal standard Local Converse Theorem for~$\G(F)$ requires twisting by representations of~$\GL_r(F)$ with~$r$ up to half the dimension of the standard representation of the dual group of~$\G$. However, if we restrict to generic supercuspidal representations of~$\G(F)$ then it can be improved when~$\G=\SO_{2N}$; we conjecture that the same is true for symplectic and odd special orthogonal groups. We also consider the possibility of using non-standard representations of the dual group to distinguish representations, giving counterexamples to possible improvements for general linear groups,~$G_2$ and~$\SO_{2N}$.
\end{abstract}

\date{\today}
\subjclass[2000]{Primary 11S70, 22E50; Secondary 11F85, 22E55.}
\keywords{Local converse problem, Jacquet's conjecture, Sharpness}
\thanks{The first-named author was supported by a grant from the Simons Foundation \#422638 and by a PSC-CUNY award, jointly funded by the
Professional Staff Congress and The City University of New York. The second-named author was supported by EPSRC grant EP/V061739/1.}
\maketitle

\section{Introduction}

Let~$F$ be a non-archimedean local field of characteristic~$0$ and residual characteristic~$p$, and let~$G=\G(F)$ be the group of points of a connected reductive group over~$F$. Given some set~$\CS$ of (equivalence classes of) irreducible (complex) representations of~$G$, a natural question to ask is:
\begin{itemize}
\item[(Q)] 
 Which invariants can we use to distinguish two elements of~$\CS$?
\end{itemize}
In the context of the local Langlands correspondence, and in this paper, the invariants we consider are \emph{twisted~$\gamma$-factors}, which are certain functions of a complex variable. 

For the group~$\G=\GL_N$ this question is now well-understood. We fix a non-trivial additive character~$\psi$ of~$F$ and let~$\CS=\Irrgen(\GL_N(F))$ be the set of irreducible generic representations of $\GL_N(F)$. Then the \emph{Local Converse Theorem for~$\GL_N$~\cite{JL18,Ch19}} says that one possible answer to (Q) is:
\begin{itemize}
\item[(A)] 
$\left\{\gamma(s, \pi \times \tau, \psi) \mid \tau\in\Irrgen(\GL_r(F)), r\le \left\lfloor\frac N2\right\rfloor\right\}$,
\end{itemize}
where the~$\gamma$-factors are those defined via Rankin--Selberg convolution~\cite{JPSS83} or by the Langlands--Shahidi method~\cite{S84}. Moreover, if~$p>N$ then this answer is optimal in the sense that the bound on~$r$ cannot be reduced, and it remains optimal if we replace~$\CS$ by the set of irreducible supercuspidal representations of~$\GL_N(F)$ (see~\cite{ALST18,A22}). Interestingly, if we look only at the set of \emph{simple} (or \emph{epipelagic}) supercuspidal representations (those of depth~$\frac 1N$) then one can do better: only twists by characters of~$\GL_1(F)$ are required (see~\cite[Proposition 2.2]{BH14}, and also~\cite{LS25} for further discussion of this).

For other groups~$\G$, analogous~$\gamma$-factors have also been defined; for example, for split classical groups over $p$-adic fields. However, the situation is not necessarily as clean as in the case of general linear groups: indeed, in full generality it \emph{should not even be possible} to distinguish irreducible representations using only~$\gamma$-factors, as we now explain. We will assume that the group~$\G$ is split and we fix a Whittaker datum for~$G=\G(F)$, so that \emph{generic} always means relative to this fixed datum.

Write~$\vG(\BC)$ for the complex dual group of~$\G$ and consider~$\varrho:\vG(\BC) \to \GL_M(\BC)$ an algebraic representation, for some integer~$M$. To each~$\pi\in\Irrgen(G)$, if we have a Langlands correspondence for~$G$ then there is a corresponding Langlands parameter~$\varphi_\pi$ and we should obtain a \emph{functorial lifting}~${}_\varrho\Pi\in\Irrgen(\GL_M(F))$ which corresponds to the Langlands parameter~$\varrho\circ\varphi_\pi$. Attached to the pair $(\pi, \varrho)$ is a conjectural gamma factor $\gamma(s, \pi, \varrho, \psi)$, predicted by the Langlands program.  Then an equality of twisted~$\gamma$-factors  for~$\pi_1,\pi_2\in\Irrgen(G)$ (relative to~$\varrho$, when these have been defined) would only guarantee~${}_\varrho\Pi_1 \cong {}_\varrho\Pi_2$, but not necessarily $\pi_1 \cong \pi_2$. This phenomenon is equivalent to the question of whether, given two generic Langlands parameters $\varphi_1,\varphi_2$ for~$G$, having~$\varrho\circ\varphi_1$ conjugate to~$\varrho\circ\varphi_2$ over~$\GL_M(\BC)$, implies that~$\varphi_1,\varphi_2$ are~$\vG(\BC)$-conjugate. There are some groups and representations for which this is true -- for example when~$\G$ is~$\Sp_{2N}$ or~$\SO_{2N+1}$ and~$\varrho$ is the standard representation -- but that is not always the case.

This question, as we allow~$\varrho$ to vary over all possible algebraic representations of~$\vG(\BC)$, is encapsulated in the notion of acceptability. First, we say that a morphism $\WD_F \rightarrow \vG(\BC)$ of the Weil-Deligne group is \emph{generic} if its adjoint $L$-function is holomorphic at $s = 1$; when the local Langlands correspondence is known, generic parameters correspond precisely to $L$-packets that have a generic element (see~\cite[\S2.2]{M24} and~\cite[Proposition~B.1]{GI16}).  \shaun{We say that two morphisms~$\varphi,\varphi':\WD_F\to\vG(\BC)$ are \emph{locally conjugate} if, for all algebraic representations~$\varrho$ of~$\vG(\BC)$, we have~$\varrho\circ\varphi\cong\varrho\circ\varphi'$; equivalently (see \cite{CG23}), 
$\varphi(w)$ and $\varphi'(w)$ are conjugate in $\vG(\BC)$ for each $w\in\WD_F$.}  Finally, we say that~$\vG(\BC)$ is 
\emph{generically $\WD_F$-acceptable}
  if any pair of \shaun{locally conjugate} 
 generic morphisms~$\WD_F\to\vG(\BC)$ are in fact conjugate \shaun{(i.e. there is a $g \in \vG(\BC)$ such that $\varphi' = \Ad(g)\circ \varphi$)}.  Then, the most general version of the local converse problem on the automorphic side, when transferred to the Galois side, requires 
 a given dual group $\vG(\BC)$ to be 
 generically $\WD_F$-acceptable.

On the other hand, there is a more general definition of acceptability, which we now recall (following \cite{L94,CG23,M24}).  We say that $\vG(\BC)$ is \emph{acceptable} if, for any \emph{finite} group $\Gamma$, any two \shaun{locally conjugate} group homomorphisms $\varphi, \varphi' : \Gamma \rightarrow \vG(\BC)$ are automatically conjugate.  Matringe~\cite{M24} has shown that if $\vG(\BC)$ is acceptable, then it is also generically $\WD_F$-acceptable.  However, it is unclear whether the reverse implication is true. For example, it is known that $\SO_{2N}(\mathbb{C})$ is unacceptable when~$N\ge 3$ (see~\cite{Yu22}) but unknown whether or not it is in general generically $\WD_F$-unacceptable. In \cite{M24}, it is shown that $\SO_{2N}(\mathbb{C})$ is generically $\WD_F$-unacceptable when $q \equiv 1 \pmod{4}$; we will show that the same is true when $q \equiv 3 \pmod{4}$.

We consider now the algebraic representation ring $\mathrm{R}[\vG]$ of~$\vG(\BC)$.  It is generated by finitely many representations (see for example~\cite[Section~23]{FH}).  For example, in the case of~$\GL_N(\BC)$ these are the exterior powers~$\Wedge^i$ of the standard representation, for~$i=1,\ldots N$.  If $\vG(\BC)$ is simply connected, these are the \emph{fundamental representations}, which are the irreducible finite-dimensional representations whose highest weight is a fundamental weight; see later for some other examples. Thus, $\WD_F$-acceptability of~$\vG(\BC)$ can be detected by only testing on a finite number of representations: if~$\varrho\circ\varphi\cong\varrho\circ\varphi'$ for these (finitely many) representations~$\varrho$ then it is also true for all algebraic representations $\rho$.

Thus, when there is a \emph{suitable local Langlands correspondence for~$\G$} (by which we mean one satisfying~\cite[Conjectures~6.1,6.5]{KT24} in the formulation of Kaletha--Taïbi, so it sends tempered representations to parameters with bounded image (\cite[Conjecture 6.1, \S 5.3]{KT24}), and there is unicity of generic representations in a tempered~$L$-packet) 
we get the following result immediately from the Local Converse Theorem for~$\GL_N$. 
To state it, we \emph{define} the twisted local factors by
\[
\gamma(s,\pi\rtimes\tau,\varrho,\psi) := \gamma(s,(\varrho\circ\varphi_{\pi})\otimes\varphi_\tau,\psi),
\]
for~$\varrho$ an algebraic representation of~$\vG(\BC)$ and~$\tau\in\Irrgen(\GL_r(F))$, though of course one would hope to have a direct definition of these~$\gamma$-factors on the automorphic side which matches with this.

\begin{lcttheorem}
Let $\G$ be a split connected reductive group over~$F$ such that the 
dual group~$\vG(\BC)$ is generically~$\WD_F$-acceptable, and for which there is a suitable local Langlands correspondence. Let~$\varrho_i$, for~$i\in I$, be generators of $\mathrm{R}[\vG]$ (which can be taken to be finite in number), and put~$K=\max_{i\in I}{\dim(\varrho_i)}$. Suppose that $\pi_1, \pi_2$ are irreducible generic tempered representations of $\G(F)$ such that
\[
\gamma(s, \pi_1\rtimes \tau,\varrho_i, \psi) = \gamma(s, \pi_2\rtimes \tau,\varrho_i, \psi)
\]
for all~$i\in I$ and all~$\tau\in\Irrgen(\GL_r(F))$ with~$r\le \left\lfloor\frac{K}{2}\right\rfloor$. Then $\pi_1 \cong \pi_2$.
\end{lcttheorem}

Note that this version of the~\lcttheoremtitle\ was proven in \cite{M24} in the case that $\vG$ is simply connected; in that case, the fundamental representations generate $\mathrm{R}[\vG]$.

Thus, in the situation of this~\lcttheoremtitle, all these twisted~$\gamma$-factors together provide an answer to~(Q), but this is almost certainly suboptimal: indeed, in the case of~$\GL_N(F)$ one does not need any of the fundamental representations other than the standard one, though one could also wonder if we could choose to cut down in a different way (see later in the introduction). This paper, then, is about the possible suboptimality of this~\lcttheoremtitle.  

Since most of our results concern groups whose dual groups have a standard representation~$\rhostd$ (faithful of smallest dimension), it is convenient to make the following definition.

\begin{defn}
Let $\G$ be a split connected reductive group over~$F$, whose dual group has a standard representation~$\rhostd$. Let~$\pi_1,\pi_2$ be irreducible generic representations of~$\G(F)$ and let~$m\ge 1$ be an integer. We say that~$\pi_1,\pi_2$ are \emph{$\gamma$-equivalent to level~$m$} if
\[
\gamma(s, \pi_1 \rtimes \tau, \rhostd, \psi) = \gamma(s, \pi_2 \rtimes \tau, \rhostd, \psi)
\]
for all irreducible supercuspidal representations $\tau$ of $\GL_r(F)$, with $1\le r\le m$. We use the same terminology for Langlands parameters also.
\end{defn}

Note that, by multiplicativity of~$\gamma$-factors, if~$\pi_1,\pi_2$ are $\gamma$-equivalent to level~$m$ then we in fact have~$\gamma(s, \pi_1 \rtimes \tau, \rhostd, \psi) = \gamma(s, \pi_2 \rtimes \tau, \rhostd, \psi)$ for all \emph{generic} irreducible representations~$\tau$ of $\GL_r(F)$, with $1\le r\le m$.

\subsection{The classical groups}

From now on, we assume that the residual characteristic of~$F$ is odd. Let $G_N$ be the group of points of a split classical group of rank $N$ -- i.e.~$\Sp_{2N}(F),\SO_{2N}(F)$, or~$\SO_{2N+1}(F)$.  The main theorem that has already been proven in the literature, which we will refer to as the \emph{the standard local converse theorem}, says the following, where we say that irreducible representations~$\pi_1,\pi_2$ of~$G_N$ are \emph{outer conjugate} if either they are conjugate (when~$G_N$ is~$\Sp_{2N}(F)$ or~$\SO_{2N+1}(F)$) or there is an element~$g\in\OO_{2N}(F)$ such that~${}^g\pi_1\cong\pi_2$ (when~$G_N=\SO_{2N}(F)$).

\begin{classicallcttheorem}[see \cite{JS03, Z18, HZ23}]
Suppose that $\pi_1, \pi_2$ are irreducible generic representations of~$G_N$, with the same central character, which are~$\gamma$-equivalent to level~$N$. Then $\pi_1$ and $\pi_2$ are outer conjugate.
\end{classicallcttheorem}

Thus, in the setting of the general~\lcttheoremtitle, this says that we \emph{only} need to consider the standard representation of the dual group.  Our first result is that we may remove the condition on the central characters of $\pi_1, \pi_2$ (see Theorem~\ref{uniformproof}).  Our subsequent focus is on the possible suboptimality of the~\classicallcttheoremtitle\ (namely, can we reduce the number of~$\GL$ twists while considering only $\varrho = \rhostd$), although in principle one could also consider generators of $R[\vG]$ other than the standard representation.  We first prove the following theorem, which shows that the~\classicallcttheoremtitle\ is optimal in the tame setting, for generic, tempered, irreducible representations of $\G_N$.

\begin{thm}[see Corollary \ref{cor:noncusp}]
Suppose~$p>N$. Then there are generic irreducible tempered (non-cuspidal) representations~$\pi,\pi'$ of~$G_N$, which are~$\gamma$-equivalent to level~$N-1$ but not outer conjugate.
\end{thm}

However, the situation changes if one assumes that $\pi_1, \pi_2$ are also supercuspidal.  We first state the main results in the setting of $\SO_{2N}$ since we have a clean statement there, and then we discuss the results for a general classical group. 

\begin{thm}[Theorem \ref{thm:SObetter} and Corollary~\ref{cor:cusp}]\label{thm:introSObetter}
Suppose~$p>N$.
\begin{enumerate}
\item If $N$ is even then the standard local converse theorem for $\SO_{2N}$ is optimal.
\item If~$N$ is odd and~$\pi_1,\pi_2$ are irreducible generic supercuspidal representations of~$\SO_{2N}(F)$ which are~$\gamma$-equivalent to level~$N-1$, then~$\pi_1, \pi_2$ are outer conjugate. Moreover, this new bound is optimal.
\end{enumerate}
\end{thm}

We are able to prove the same results for the other classical groups \emph{conditional on a conjecture} which we now explain. First, we recall that a discrete Langlands parameter for a classical group is a multiplicity-free direct sum of twists of irreducible self-dual representations of $W_F$ by irreducible representations of~$\SL_2(\BC)$ (see the proof of Lemma \ref{lem:nonmaxbetter} for a more precise statement).  It therefore becomes natural to understand if there is any possible improvement in the standard local converse theorem for~$\GL_N$ when one restricts to the setting of self-dual supercuspidal representations.  While the standard local converse theorem for $\GL_N$ asserts that the bound $[N/2]$ is optimal for general supercuspidal representations of $\GL_N$, we make the following conjecture.

\begin{conj}(see Conjecture \ref{conj:better})\label{conj:betterintro}
Suppose~$p>N$ and~$N$ is odd. Let~$\pi,\pi'$ be self-dual supercuspidal representations of~$\GL_{2N}(F)$ of the same parity which are~$\gamma$-equivalent to level~$N-1$.
Then~$\pi\cong\pi'$.
\end{conj}

We have no evidence as to whether the assumptions on $p$ or $N$ are needed, or even whether the above bound is optimal; at best, it can be lowered by $1$ more, as the following shows:

\begin{thm}(see Theorem \ref{thm:better})
Suppose~$p>N$, set~$M=2\lfloor\frac N2\rfloor$, and fix a parity. There exist inequivalent self-dual supercuspidal representations~$\pi,\pi'$ of~$\GL_{2N}(F)$ of this parity which are~$\gamma$-equivalent to level~$M-2$.
\end{thm}

Assuming Conjecture \ref{conj:betterintro}, we also have the analogue of Theorem~\ref{thm:introSObetter} for the symplectic and odd special orthogonal groups. We refer to Theorem \ref{thm:Gbetter} and Corollary \ref{cor:cusp} for the statement.

%

We end this section by considering the question of acceptability. In Example~\ref{so6}, we show that $\SO_6$ is generically $WD_F$-unacceptable when $q \equiv 3\pmod{4}$.  To be more specific, the representation ring~$R[\SO_6(\BC)]$ is generated by~$\Wedge^1,\Wedge^2,\Wedge^3_+,\Wedge^3_-$, where~$\Wedge^3=\Wedge^3_+\oplus\Wedge^3_-$, and we give an example of inequivalent generic (supercuspidal) parameters $\varphi_1, \varphi_2 : W_F \rightarrow \SO_6(\mathbb{C})$ which are conjugate under~$\OO_6(\BC)$ (so that their compositions with~$\Wedge^r$ are equivalent, for any~$r$) and, moreover~$\Wedge^3_+\circ\varphi_1\cong\Wedge^3_-\circ\varphi_2$. In particular, this means that all their twisted~$\gamma$-factors will be equal, whatever representation of the dual group we take. As we comment in Remark~\ref{rmk:SO2odd}, the same method could be used to give discrete parameter counterexamples for~$\SO_{2N}$ whenever~$N\ge 3$ is odd.

Moreover, we are able to leverage this example to show that~$\SO_{2N}(\mathbb{C})$ is generically $\WD_F$-unacceptable when $q \equiv 3 \pmod{4}$ (see~Corollary~\ref{cor:SO2N} ); this complements Matringe's result~\cite[Proposition~7.3]{M24} that~$\SO_{2N}(\mathbb{C})$ is generically $\WD_F$-unacceptable when $q \equiv 1 \pmod{4}$. Note, however, that the parameters demonstrating these results are \emph{not} discrete.

\subsection{The exceptional group $\G_2$}

We also investigate the exceptional split group of type $\G_2$, when~$p>3$.  In this case there are two fundamental representations of $\G_2(\mathbb{C})$: the standard $7$-dimensional, and the adjoint, which is~$14$-dimensional.  We begin by proving an expected result, the \emph{standard local converse theorem for $\G_2$}, which says that the family of $\GL_i$-twisted gamma factors, $i = 1, 2, 3$, using only the standard representation of the dual group, determines an irreducible generic representation of $\G_2$.  We then prove that one cannot lower this bound of twisting to $i = 1,2$, and that even if we include the adjoint representation of the dual group, we still do not obtain a local converse theorem.  To summarize, we prove the following theorem (see Theorem \ref{thm:converseG2}).

\begin{thm} Suppose~$p>3$.
\begin{enumerate}
\item Let $\varphi_1, \varphi_2$ be two generic Langlands parameters for $\G_2(F)$.  If~$\varphi_1,\varphi_2$ are $\gamma$-equivalent to level~$3$, then $\varphi_1 \cong \varphi_2$. 
\item There are inequivalent generic supercuspidal representations~$\pi_1,\pi_2$ of~$\G_2(F)$ such that
\[
\gamma(s,\pi_1\rtimes\tau,\varrho,\psi) := \gamma(s,\pi_2\rtimes\tau,\varrho,\psi) ,
\]
for~$\varrho$ either fundamental representation of~$\G_2(\BC)$, and all supercuspidal representations of~$\GL_r(F)$, with~$1\le r\le 2$; in particular, they are~$\gamma$-equivalent to level~$2$.
\end{enumerate}
\end{thm}

\subsection{The general linear groups $\GL_N$}

Finally, we investigate whether one can improve upon the~\lcttheoremtitle\ for $\GL_N$ in ways other than by cutting out the non-standard fundamental representations.  In fact, in~\cite[Conjecture~1.2]{YZ22}, a hint was provided along these lines.  The authors state a conjecture, attributed to Ramakrishnan~\cite{Ram94}, that for unitarizable supercuspidal representations of~$\GL_N(F)$ one only needs to consider the~$\GL_1$-twisted gamma factors while allowing all fundamental representations.  Namely, the gamma factors
\[
\gamma(s, \pi\rtimes \eta,\Wedge^i, \psi),\text{ for~$i=1,\ldots,\left\lfloor\tfrac N2\right\rfloor$ and~$\eta$ a character of~$F^\times$}.
\]
(They also assume that the central character is fixed, though that is unnecessary since the central character is determined by the standard twisted~$\gamma$-factors~$\gamma(s, \pi\times \eta,\psi)$; see, for example,~\cite[Corollary~2.7]{JNS15}.) In Section~\ref{sec:ramakrishnan}, we prove that this conjecture is false in general:

\begin{thm} 
Suppose~$p$ is odd.  Then there are inequivalent irreducible parameters~$\varphi_1,\varphi_2:W_F\to\GL_4(\BC)$, corresponding to unitarizable supercuspidal representations of $\GL_4(F)$, such that
\[
\gamma(s, (\Wedge^i\circ\varphi_1)\otimes\eta, \psi) = \gamma(s, (\Wedge^i\circ\varphi_2)\otimes\eta, \psi),
\]
for all characters~$\eta$ of~$W_F$ and~$i=1,\ldots,4$.
\end{thm}
Thus twisting only by characters of~$F^\times$ cannot be sufficient to distinguish supercuspidal representations, even if all fundamental representations of the dual group are considered.

\subsection*{Acknowledgements}
We would like to thank Wee Teck Gan, Baiying Liu, Nadir Matringe, A. Raghuram, and Shuichiro Takeda for helpful conversations.
We also thank the referee for their careful reading of the paper, and for several observations which have improved it.

\color{black}
\section{Representations of the Weil group and~$\gamma$-factors}\label{llp}

We recall the parametrization of tame Langlands parameters and their associated local factors. Let $F$ be a non-archimedean local field of characteristic zero, with ring of integers~$\CO_F$, maximal ideal~$\CP_F$, residue field~$k_F\simeq\BF_q$, where~$q =q_F= p^f$ is a power of the residual characteristic~$p$, and~$\mu'_F$ the group of roots of unity of order prime to~$p$. We use analogous notation for any extension~$E$ of~$F$. We also fix a non-trivial additive character $\psi_F$ of $F$ of level one: that is, it is trivial on $\CP_F$ but not on $\CO_F$. For any finite extension $E/F$, write~$e_{E/F}$ and~$f_{E/F}$ for the ramification index and residue degree of~$E/F$ respectively. We also define an additive character $\psi_E$ of $E$ via $\psi_E=\psi_F \circ \tr_{E/F}$; if $E/F$ is tamely ramified (that is,~$p\nmid e_{E/F}$) then $\psi_E$ is also of level one.

Let~$\overline{F}$ be a fixed algebraic closure of~$F$. We normalize the additive valuation~$\val$ on~$F$ to have image~$\BZ\cup\{\infty\}$ and extend this to an additive valuation on~$\overline F$; thus, for~$E/F$ any finite extension, we have~$\val(E^\times)=\frac 1{e_{E/F}}\BZ$. For~$r$ a real number, we write
\[
\CP_E^r = \{x\in E : \val(x)\ge r\},\qquad
\CP_E^{r+} =  \{x\in E : \val(x) > r\}
\]
so, for example,~$\CP_E^{1/e_{E/F}}=\CP_E=\CP_E^{0+}$. We put~$U_E=U_E^0=\CO_E^\times$, with filtration subgroups~$U_E^r=1+\CP_E^r$, for~$r>0$, and~$U_E^{r+}=1+\CP_E^{r+}$, for~$r\ge 0$. 

We also write~$|\cdot|$ for the normalized absolute value on~$\overline{F}$, given by
\[
|x| = q^{-\val(x)}, \qquad\text{ for }x\in \overline{F}^\times.
\]
Note that if~$E/F$ is a finite extension then the usual normalization of the absolute value on~$E$ is given by~$x\mapsto |x|^{[E:F]}$, for~$x\in E$.

Let~$W_F$ be the Weil group of~$F$. Implicit in the definition of the Weil group are compatible choices of isomorphisms~$E^\times\simeq W_E^{\text{ab}}$, for each finite extension~$E/F$, via which we will identify the quasi-characters of~$E^\times$ and of~$W_E$. Moreover, if $F \subseteq E \subseteq K$ and $\chi$ is a quasi-character of $E^{\times}$, let $\chi_K:=\chi \circ \N_{K/E}$. When~$\chi$ is viewed as a quasi-character of $W_E$, the quasi-character~$\chi_K$ is simply the restriction of $\chi$ to the subgroup~$W_K$. For~$E/F$ a finite extension, write~$\Ind_{E/F}=\Ind_{W_E}^{W_F}$ for the induction functor.

\subsection{Admissible pairs}

Let~$N>1$ be an integer not divisible by~$p$. In this situation, there is a nice parametrization of the irreducible~$N$-dimensional representations of $W_F$ in terms of admissible quasi-characters, introduced by Howe.

\begin{defn}[Howe~\cite{Ho77}]
An \emph{admissible pair} of degree~$N$ is a pair~$(E/F,\chi)$ where $E/F$ is a (tamely ramified) extension of degree $N$ and~$\chi$ is a quasi-character of $E^{\times}$ such that
\begin{enumerate}
\item $\chi$ does not come via the norm from a proper subfield of $E$ containing $F$; and
\item if the restriction $\chi|_{U_E^{0+}}$ comes via the norm from a subfield $F \subseteq L \subseteq E$, then $E/L$ is unramified.
\end{enumerate}
Two admissible pairs~$(E_1/F,\chi_1)$ and~$(E_2/F,\chi_2)$ are said to be \emph{conjugate} if there is an $F$-isomorphism from $E_1$ to $E_2$ which takes $\chi_1$ to $\chi_2$.
\end{defn}
In this situation we also say that~$\chi$ is an \emph{admissible quasi-character} of~$E^\times$ (relative to~$F$). For any quasi-character of~$E^\times$, we write
\[
\rho_\chi:=\Ind_{E/F}\chi.
\]

\begin{thm}[{Moy~\cite[Theorem 2.2.2]{M86}}]\label{thm:moy1}
If~$(E/F,\chi)$ is an admissible pair of degree~$N$, then $\rho_\chi$ is an irreducible $N$-dimensional representation of $W_F$. Furthermore, two admissible pairs induce to equivalent representations if and only if they are conjugate, and each irreducible $N$-dimensional representation of $W_F$ is induced from an admissible pair.
\end{thm}

Let $E/F$ be a (tamely ramified) extension of degree $N$, and $\chi$ a \emph{ramified} quasi-character of~$E^{\times}$; that is,~$\chi$ is non-trivial on~$U_E$. Let $d(\chi)\in \frac 1{e_{E/F}}\BZ$ be the \emph{normalized depth} of~$\chi$: that is,~$\chi$ is trivial on $U_E^{d(\chi)+}$ but not on $U_E^{d(\chi)}$. (The \emph{conductoral exponent} of~$\chi$ is then~$e_{E/F} d(\chi)+1$.) There is then~$c_{\chi} \in \CP_E^{-d(\chi)}$, well-defined modulo~$\CP_E^{- {d(\chi)}/2}$ such that
\[
\chi(1+x)=\psi_E(c_{\chi}x),\qquad\text{ for }x\in\CP_E^{(d(\chi)/2)+}.
\]
We say that~$c_\chi$ \emph{represents}~$\chi$. We extend this to unramified quasi-characters~$\chi$ of~$E^\times$ by setting~$d(\chi)=0$ and~$c_\chi=0$ (which we will also take for quasi-characters of depth~$0$).

Note that if~$K/E$ is tamely ramified then~$d(\chi_K)=d(\chi)$, since~$\N_{K/E}(U_K^r)=U_E^r$, for any~$r>0$. Moreover~$\chi_K$ is also represented by~$c_\chi$: that is, we can choose~$c_{\chi_K} = c_{\chi}$. We also make the following simple observation.

\begin{lem}\label{lem:depth}
Let~$E/F$ and~$L/F$ be finite tamely ramified extensions and let~$\chi,\eta$ be quasi-characters of~$E^\times, L^\times$ respectively, represented by~$c_\chi$ and~$c_\eta$ respectively. We consider the quasi-character~$\chi_{EL}\eta_{EL}$ of the compositum~$(EL)^\times$.
\begin{enumerate}
\item We have~$d(\chi_{EL}\eta_{EL})\le \max\{d(\chi),d(\eta)\}$.
\item If~$d(\chi)\ne d(\eta)$ then~$\chi_{EL}\eta_{EL}$ is ramified of normalized depth~$d(\chi_{EL}\eta_{EL})=\max\{d(\chi),d(\eta)\}$, and is represented by~$c_\chi+c_\eta$.
\end{enumerate}
\end{lem}

There are certain situations in which we can tell directly from a representative for a quasi-character~$\chi$ of~$E^\times$ that it is admissible over~$F$. For~$d<0$ a rational number, we say that~$\beta\in E$ is \emph{quasi-minimal (over F) of depth~$d$} if $\val(\beta)=-d$ and
\begin{itemize}
\item every element of the coset~$\beta+\CP_E^{-d/2}$ generates the (tame) extension~$E/F$.
\end{itemize}
Note that~$\beta$ is called \emph{minimal} in~\cite[(1.4.14)]{BK93} if $\val(\beta) = -d$ and every element of the coset~$\beta+\CP_E^{-d+}$ generates~$E/F$. We will use the notion especially in the \emph{minimal totally ramified} case, that is when~$E/F$ is totally ramified of degree~$N$ and~$d$ has least denominator~$N$.

If~$\beta$ is quasi-minimal then the character~$1+x\mapsto \psi_E(\beta x)$ of~$U_E^{(d/2)+}$ does not come via the norm from a proper subfield of $E$ containing $F$. In particular, any quasi-character represented by a quasi-minimal element is admissible; this generalizes the remark that any \emph{generic} quasi-character (in the sense of Kutzko~\cite{Ku80}, see~\cite[Definition~2.2.3]{M86}) is admissible.

\subsection{Local factors}

As above, let~$E/F$ be an extension of degree~$N$ not divisible by~$p$, let~$\chi$ be a ramified quasi-character of~$E^\times$ of depth~$d(\chi)> 0$, and set~$J = U_E^{d(\chi)/2}$ and $H = U_E^{(d(\chi)/2)+}$. We define the Gauss sum 
\[
G(\chi,\psi_F)=[J:H]^{-\frac{1}{2}} \sum_{x \in J/H} \chi^{-1}(x)\psi_F(c_{\chi}(x-1)).
\]
This collapses to~$1$ if~$J=H$, which happens precisely when the integer~$e_{E/F}d(\chi)$ is odd.

Associated to any quasi-character~$\chi$ of~$E^\times$ we also have local factors~$L(s,\chi)$,~$\epsilon(s,\chi,\psi_E)$ and~$\gamma(s,\chi,\psi_E)$, by Tate's thesis~\cite{T50} (see also the account in~\cite[\S23]{BH06}). When~$\chi$ is ramified we have~$L(s,\chi)=1$ and~$\epsilon(s,\chi,\psi_E)=\gamma(s,\chi,\psi_E)$. Denote by $\epsilon(\chi, \psi_E)$ the value of the $\epsilon$-factor at $s=0$. Note that $\epsilon(s, \chi, \psi_E)=\epsilon(\chi \lvert \cdot \rvert^{Ns}, \psi_E)$; the extra factor~$N$ appears here because of our normalization of the absolute value.

\begin{prop}[{Moy, \cite[(2.3.17)]{M86}}]\label{prop:epsilonfactorformula}
For~$\chi$ a ramified quasi-character of~$E^\times$, we have
\[
\epsilon(\chi, \psi_E) = \chi^{-1}(c_{\chi}) \psi_E(c_{\chi}) |c_{\chi}|^{N/2} G(\chi,\psi_E).
\]
\end{prop}

As an immediate corollary, we get the following simple result.

\begin{cor}[{cf.~\cite[Lemma~4.3]{A22}}]\label{cor:lem2}
Suppose~$\chi_1,\chi_2$ are ramified quasi-characters of~$E^\times$ of normalized depth~$d\ge 0$, and that~$\chi_1,\chi_2$ coincide on~$U_E^{d/2}$ so that they can be represented by a common element~$c_\chi$. Then
\[
\gamma(s, \chi_1, \psi_E) = 
\left(\frac{\chi_1(c_\chi)}{\chi_2(c_\chi)}\right)\gamma(s, \chi_2, \psi_E).
\]
\end{cor}

More generally, for~$\rho$ any semisimple finite-dimensional complex representation of the Weil group~$W_F$, there are local factors~$L(s,\rho)$,~$\epsilon(s,\rho,\psi_F)$ and~$\gamma(s,\rho,\psi_F)$ (see, for example,~\cite[Theorem~29.4]{BH06}). We have~$L(s,\rho)=1$ when~$\rho$ is irreducible of dimension~$N>1$ so the~$\epsilon$ and~$\gamma$ factors again coincide. Moreover, the local factors are multiplicative:~$\gamma(s,\rho_1\oplus\rho_2,\psi_F)=\gamma(s,\rho_1,\psi_F)\gamma(s,\rho_2,\psi_F)$, and similarly for~$L$ and~$\epsilon$.

We end this section by recalling the following results on the tensor product of such representations. For~$K/F$ a finite extension, we write~$\lambda_{K/F}(\psi_F)$ for the Langlands constant associated to $K/F$ and $\psi_F$ (see \cite[Theorem~29.4 and~\S34.3]{BH06}).

\begin{lem}[{\cite[Lemmas~2.4,~2.5]{ALST18}}]\label{lem1}
Let~$E$ and~$L$ be tamely ramified field extensions of~$F$, and let~$\chi$ and~$\eta$ be quasi-characters of~$E^\times$ and~$L^\times$ respectively. For~$g\in W_F$ we write~$E_g=g(E)$ and~$K_g=E_gL$, and set~$\theta_g=({}^g\chi)_{K_g} \eta_{K_g}$. Then
\[
\rho_\chi\otimes \rho_\eta \cong \bigoplus_{g\in W_L\backslash W_F/W_E} \Ind_{K_g/F} \theta_g
\]
and
\[
\gamma(s, \rho_\chi  \otimes \rho_\eta, \psi_F) = \prod_{g\in W_L\backslash W_F/W_E}
\lambda_{K_g/F}(\psi_F)\gamma(s,\theta_g, \psi_{K_g}).
\]
\end{lem}

\subsection{Totally ramified quasi-characters}

In this subsection, we give the main technical results on equalities of~$\gamma$-factors for representations of~$W_F$ which are at the heart of our sharpness examples. We continue with~$E/F$ an extension of degree~$N$ not divisible by~$p$ and let~$\chi$ be a ramified quasi-character of~$E^\times$ of depth~$d(\chi)> 0$. We say that~$\chi$ is \emph{totally ramified} if~$E/F$ is totally ramified, and that~$\chi$ is \emph{minimal totally ramified} if also it is represented by a minimal element (equivalently, $d(\chi)$ has least denominator~$N$).

For~$L/F$ a finite extension of degree~$r$ and~$\alpha\in L$, we write~$f_\alpha(X)\in F[X]$ for the characteristic polynomial of~$\alpha$ acting by multiplication on~$L$, so that
\[
f_\alpha(X) = \prod_{\sigma\in\Gal(L/F)} (X-\sigma(\alpha)),
\]
where~$\Gal(L/F)$ is the set of~$F$-embeddings of~$L$ into~$\overline{F}$. Of course, this polynomial~$f_\alpha$ depends on the field~$L$ as well as~$\alpha$ but the field will be clear from context. 

\begin{prop}\label{prop:basic}
For~$i$ in some finite set~$I$, let~$(E_i/F,\chi_i),(E_i/F,\chi'_i)$ be totally ramified admissible pairs of degree~$N_i$ and depth~$d_i$, such that~$\chi_i,\chi'_i$ coincide on~$U^{d_i/2}_{E_i}$; choose~$\beta_i$ which represents both. Suppose~$(L,\eta)$ is an admissible pair of degree~$r$, represented by~$\alpha$, with depth distinct from all~$d_i$, and set~$\tau=\Ind_{L/F}\eta$. Then the following are equivalent:
\begin{enumerate}
\item\label{prop:basic.i} $\gamma(s,(\bigoplus\limits_{i\in I}\rho_{\chi_i})\otimes\tau,\psi_F)=\gamma(s,(\bigoplus\limits_{i\in I}\rho_{\chi'_i})\otimes\tau,\psi_F)$;
\item $\prod\limits_{i\in I} \chi_i((-1)^r f_\alpha(-\beta_i)) = \prod\limits_{i\in I} \chi'_i((-1)^r f_\alpha(-\beta_i))$.
\end{enumerate}
\end{prop}
Note that the condition on the depth of~$\eta$ in Proposition~\ref{prop:basic} is implied if~$r$ is not divisible by any least denominator of~$d_i$, since~$d(\eta)\in\frac 1r\mathbb{Z}$. 

\begin{proof}
The ideas of the proof are essentially in the proof of~\cite[Theorem~3.2]{A22}. For~$g\in W_F$ and~$i\in I$, we write~$E_{i,g}=g(E_i)$ and~$K_{i,g}=E_{i,g}L$, and set~$\theta_{i,g}=({}^g\chi_i)_{K_{i,g}} \eta_{K_{i,g}}$ and~$\theta'_{i,g}=({}^g\chi'_i)_{K_{i,g}} \eta_{K_{i,g}}$. By Lemma~\ref{lem:depth}(ii), the quasi-characters~$\theta_{i,g}$ and~$\theta'_{i,g}$ are represented by~$g(\beta_i)+\alpha$ and have depth~$d_{i,g}:=\max\{d_i,d(\eta)\}$. Thus, by Lemma~\ref{lem1} and multiplicativity, we have
\[
\gamma(s,({\textstyle\bigoplus\limits_{i\in I}\rho_{\chi_i}})\otimes\tau,\psi_F)= \prod_{i\in I} \prod_{g\in W_L\backslash W_F/W_{E_i}}
\lambda_{K_{i,g}/F}(\psi_F)\gamma(s,\theta_{i,g}, \psi_{K_{i,g}}),
\]
Thus the equality in~\ref{prop:basic.i} is equivalent to
\[
\prod_{i\in I} \prod_{g\in W_L\backslash W_F/W_{E_i}} \gamma(s,\theta_{i,g}, \psi_{K_{i,g}}) = \prod_{i\in I} \prod_{g\in W_L\backslash W_F/W_{E_i}} \gamma(s,\theta'_{i,g}, \psi_{K_{i,g}}).
\]
Moreoever,~$\theta_{i,g}$ and~$\theta'_{i,g}$ coincide on~$U^{d_{i,g}/2}_{K_{i,g}}$ since~$\N_{K_{i,g}/E_{i,g}}(U^{d_{i,g}/2}_{K_{i,g}})\subseteq U^{d_i/2}_{E_{i,g}}$, where the quasi-characters~${}^g\chi_i$ and~${}^g\chi'_i$ coincide by hypothesis. Thus it follows from Corollary~\ref{cor:lem2} that this is equivalent to
\[
\prod_{i\in I} \prod_{g\in W_L\backslash W_F/W_{E_i}} \theta_{i,g} (g(\beta_i)+\alpha) = 
\prod_{i\in I} \prod_{g\in W_L\backslash W_F/W_{E_i}} \theta'_{i,g} (g(\beta_i)+\alpha),
\]
which is in turn equivalent to
\[
\prod_{i\in I} \prod_{g\in W_L\backslash W_F/W_{E_i}} ({}^g\chi_i)_{K_{i,g}} (g(\beta_i)+\alpha) = 
\prod_{i\in I} \prod_{g\in W_L\backslash W_F/W_{E_i}} ({}^g\chi'_i)_{K_{i,g}} (g(\beta_i)+\alpha).
\]
Thus, to complete the proof, it suffices to show that
\[
\prod_{g\in W_L\backslash W_F/W_{E_i}} ({}^g\chi_i)_{K_{i,g}} (g(\beta_i)+\alpha) = \chi_i((-1)^r f_\alpha(-\beta_i)),
\]
for which we drop the subscripts~$i$. As in~\cite[after~(4.3)]{A22}, we have
\[
\prod_{g\in W_L\backslash W_F/W_E} ({}^g\chi)_{K_{g}} (g(\beta)+\alpha) = \chi\left( \prod_{g\in W_L\backslash W_F/W_E} \prod_{\sigma\in\Gal(K_g/E_g)}(\beta+g^{-1}\sigma(\alpha))\right).
\]
But, as also explained in~\cite{A22}, this further simplifies to
\[
\chi\left(\prod_{\sigma\in\Gal(L/F)}(\beta+\sigma(\alpha))\right) = \chi((-1)^r f_\alpha(-\beta)).
\]
\end{proof}

We will apply this lemma to the case where all~$N_i=N$ are equal and all depths~$d_i$ are equal, since in that case one can further simplify. To this end, we introduce some further notation. Let~$(E/F,\chi)$ be a totally ramified admissible pair of degree~$N$ and depth~$d$ represented by~$\beta$, and let~$M$ be the least denominator of~$d$. For~$L/F$ an extension of degree less than~$M$ and~$\alpha\in L$ we set
\[
u^+_\alpha(X)= \N_{L/F}(-\alpha)^{-1} f_\alpha(-X), \qquad u^-_\alpha(X) = (-X)^{-r} f_\alpha(-X).
\]
Since~$u^\pm_\alpha(X)$ are Laurent polynomials with coefficients in~$F$, we have~$u^\pm_\alpha(\beta)\in E$. Moreover, since~$\val(\alpha)\ne\val(\beta)$, we see that
\[
\begin{cases} u^+_\alpha(\beta) \in U_E^{0+} &\text{ if }\val(\alpha)<\val(\beta),\\
u^-_\alpha(\beta) \in U_E^{0+} &\text{ if }\val(\alpha)>\val(\beta). 
\end{cases}
\]
Now put
\begin{align*}
t^+_{\beta,L} &= \inf\left\{ \val\left(1-u^+_\alpha(\beta))\right) \mid \alpha\in L,\ \val(\alpha)<\val(\beta)\right\}, \\[5pt]
t^-_{\beta,L} &= \inf\left\{ \val\left(1- u^-_\alpha(\beta))\right) \mid \alpha\in L,\ \val(\beta)<\val(\alpha)<0\right\},
\end{align*}
where we interpret~$t^-_{\beta,L}=\infty$ if the set in its definition is empty. Since~$\val(E)=\frac 1N\mathbb{Z}$, we have~$t^\pm_{\beta,L}\ge \frac 1N$. 

Now, for~$r\le M$ an integer, we set
\[
t_\beta(r) = \min\left\{ t^+_{\beta,L}, t^-_{\beta,L} \mid [L:F]< r\right\}.
\]
Again, we have~$t_\beta(r)\ge \frac 1N$. We will be particularly interested in cases when this inequality is strict. As a simple case, we have
\begin{equation}\label{eqn:tb2}
t_\beta(2) = \min\left\{\left|d-\tfrac nM\right| \mid n\in\BZ\right\},
\end{equation}
since we only need to consider~$\alpha\in F$, where~$u^\pm_\alpha(X)=1+(\alpha^{-1}X)^{\pm 1}$. More generally, we have the following, which will need later:

\begin{lem}\label{lem:tbeta}
Suppose~$E/F$ is totally ramified of degree~$N$ and~$d=\frac mM$ has least denominator~$M$. Suppose~$\beta\in E$ has valuation~$-d$ and let~$r$ be the least integer such that~$rm\equiv \pm 1\pmod M$. If $r > 1$, then ~$t_\beta(r)>\frac 1M$.
\end{lem}

\begin{proof}
It is enough to prove that, for any integer~$s < r$
, any extension~$L/F$ of degree~$s$, and any~$\alpha\in L$, we have~$u^{\pm}_\alpha(\beta)\in U_E^{(1/M)+}$. Now~$u^{\pm}_\alpha(X)$ is a polynomial of degree~$s$ in~$X^{\pm 1}$ with coefficients in~$F$ and constant term~$1$. Thus, for each~$1\le i\le s$, the valuation of the monomial of degree~$i$ when evaluated at~$\beta$ lies in~$\pm \frac{mi}M +\BZ$, and it is greater than~$0$ by construction. In particular, these all have distinct valuations (as~$s<M$) so the valuation of~$u^{\pm}_\alpha(\beta)-1$ is exactly the minimum of the valuations of these monomials. Since none of these monomials have valuation in~$\frac 1M+\BZ$ (by the minimality of~$r$), they are each at least~$\frac 2M$. Thus the valuation of~$u^{\pm}_\alpha(\beta)-1$ is at least~$\frac 2M$, so strictly greater than~$\frac 1M$, so~$u^{\pm}_\alpha(\beta)\in U_E^{(1/M)+}$.
\end{proof}

The following result is at the root of all our sharpness examples, and generalizes the essence of the proof of sharpness in the local converse theorem for~$\GL_N$ in~\cite{ALST18,A22}.

\begin{prop}\label{prop:basic2}
Let~$d>0$ be a rational number with least denominator~$M$ and let~$r\le \min\{p,M\}$ be a natural number. 
For~$i\in I$, let~$(E_i/F,\chi_i),(E_i/F,\chi'_i)$ be totally ramified admissible pairs of degree~$N$ and depth~$d$, such that~$\chi_i,\chi'_i$ coincide on~$U^{d/2}_{E_i}$; choose~$\beta_i$ which represents both. Set~$t_i=t_{\beta_i}(r)$, for~$i\in I$, and suppose that:
\begin{enumerate}
\item\label{prop:basic2.i} $\chi_i,\chi'_i$ coincide on restriction to~$U_{E_i}^{t_i}$, for each~$i\in I$;\\[-8pt]
\item\label{prop:basic2.ii} $\prod_{i\in I} \chi_i(\beta_i) = \prod_{i\in I} \chi'_i(\beta_i)$;\\[-8pt]
\item\label{prop:basic2.iii} $ \prod_{i\in I} \chi_i$ and~$\prod_{i\in I} \chi'_i$ coincide on restriction to~$F^\times$.
\end{enumerate}
Then, for~$\tau$ any irreducible representation of~$W_F$ of dimension less than~$r$, we have
\[\textstyle
\gamma(s,(\bigoplus\limits_{i\in I} \rho_{\chi_i})\otimes\tau,\psi_F) = \gamma(s,(\bigoplus\limits_{i\in I} \rho_{\chi'_i})\otimes\tau,\psi_F).
\]
\end{prop}

Note that if~$d$ is chosen to have least denominator~$N$ then we are considering minimal totally ramified admissible pairs in Proposition~\ref{prop:basic2}.

\begin{proof}
Suppose~$\tau$ is an irreducible representation of~$W_F$ of dimension less than~$r$. Since~$p\ge r > \dim(\tau)$, we have~$p\nmid \dim(\tau)$ so Theorem~\ref{thm:moy1} implies that~$\tau\cong \Ind_{L/F} \eta$, for some admissible pair~$(L/F,\eta)$. Write~$d(\eta)$ for its normalized depth and~$\alpha\in L$ for some element representing~$\eta$. Since~$\dim(\tau)<r\le M$ we have~$d(\eta)\ne d$ so, by Proposition~\ref{prop:basic}, we need only check that
\[
\prod\limits_{i\in I} \chi_i((-1)^r f_\alpha(-\beta_i)) = \prod\limits_{i\in I} \chi'_i((-1)^r f_\alpha(-\beta_i)).
\]
If~$d(\eta)<d$ then either~$\alpha=0$ or~$0>\val(\alpha)>\val(\beta_i)$ for each~$i\in I$. In the former case,~$f_\alpha(X)=X^r$ so the equality follows from~\ref{prop:basic2.ii}, while in the latter case we note that~$(-1)^r f_\alpha(-\beta_i) = \beta_i^r u^-_\alpha(\beta_i)$ and the equality follows from~\ref{prop:basic2.ii} together with~\ref{prop:basic2.i}.

On the other hand, if~$d(\eta)>d$ then $\val(\alpha)<\val(\beta_i)$ for each~$i\in I$. This time we write~$(-1)^r f_\alpha(-\beta_i) = \N_{L/F}(\alpha) u^+_\alpha(\beta_i)$ so the equality follows from~\ref{prop:basic2.iii} together with~\ref{prop:basic2.i}.
\end{proof}

\subsection{Self-dual representations and determinants}

We recall some results on self-duality of representations of~$W_F$ and determinants. We assume throughout that this subsection that~$p\ne 2$. Firstly, we have the following result due to Moy.

\begin{prop}[{\cite[Theorem~1]{M84}}]\label{prop:moy}
Let~$(E/F,\chi)$ be an admissible pair with~$E\ne F$. Then the irreducible representation~$\rho_\chi$ is self-dual if and only if there is an intermediate field~$E\supset L\supseteq F$ such that~$E/L$ is quadratic and, writing~$\sigma$ for the non-trivial element of~$\Gal(E/L)$, we have~${}^\sigma\chi=\chi^{-1}$. In that case,
\begin{enumerate}
\item\label{prop:moy.i} $\rho_\chi$ is orthogonal if and only if~$\chi|_{L^\times}$ is trivial;
\item\label{prop:moy.ii} $\rho_\chi$ is symplectic if and only if~$\chi|_{L^\times}$ is non-trivial.
\end{enumerate}
\end{prop}
We call an admissible pair~$(E/F,\chi)$ as in Proposition~\ref{prop:moy} \emph{self-dual}, and also call it symplectic/orthogonal according to the parity of~$\rho_\chi$.

We note that, since~$p\ne 2$, the only odd-dimensional self-dual irreducible representations of~$W_F$ are the characters of order dividing~$2$ (all orthogonal). We also observe that there are many even-dimensional self-dual irreducible representations of~$W_F$ of both types. Indeed, suppose~$E/F$ is an extension of degree~$2N$ with an involution~$\sigma$ whose fixed field is~$L$. For~$d\ge 0$ we write~$\overline U_E^{d+}=\{x\in U_E^{d+}\mid \N_{E/L}(x)=1\}$.

Let~$\beta\in E$ be quasi-minimal of depth~$d>0$ such that~$\sigma(\beta)=-\beta$, so that every quasi-character of~$E^\times$ represented by~$\beta$ is admissible. We write~$\Xi_\beta$ for the set of characters of the pro-$p$-group~$U_E^{0+}$ which are represented by~$\beta$, so their restrictions to~$U_E^{(d/2)+}$ are all equal; write~$\psi_\beta$ for this restriction. Then~$\sigma$ acts on~$\Xi_\beta$ by $(\sigma\chi)(x)=\chi(\sigma(x)^{-1})$, and the Glauberman correspondence gives a bijection between the set of fixed points~$\Xi_\beta^\sigma$ and the set of characters of~$\overline U_E^{0+}$ whose restriction to~$\overline U_E^{(d/2)+}$ is~$\psi_\beta|_{\overline U_E^{(d/2)+}}$. Thus~$\Xi_\beta^\sigma$ has size~$[\overline U_E^{0+}:\overline U_E^{(d/2)+}]$, which is always non-zero, and bigger than~$1$ (so at least~$p$) as soon as~$d\ge \frac 2{e_{E/F}}$.

Notice that each~$\chi\in \Xi_\beta^\sigma$ restricts trivially to~$U_L^{0+}$ (since~$p\ne 2$) so we can extend~$\chi$ to~$L^\times U^{0+}_E$ by requiring that~$\chi|_{L^\times}$ is either trivial or the class field character~$\omega_{E/L}$. There are then exactly~$[E^\times:L^\times U^{0+}_E]$ further extensions to a character~$\chi$ of~$E^\times$, and we have~${}^\sigma\chi=\chi^{-1}$ for all such extensions since~$\N_{E/L}(E^\times)\subset L^\times U^{0+}_E$. When~$E/L$ is ramified, there are exactly two such extensions, and their quotient is the unramified character of~$E^\times$ of order~$2$; that is, they differ only by a sign on a uniformizer of~$E$. We summarize this discussion:

\begin{lem}\label{lem:exist}
Let~$E/F$ be an extension of degree~$2N$ with involution~$\sigma$ and let~$\beta\in E$ be quasi-minimal of depth~$d>0$ such that~$\sigma(\beta)=-\beta$. Suppose~$d>\frac 2{e_{E/F}}$. 
\begin{enumerate}
\item\label{lem:exist.i} There are at least~$p$ characters~$\chi$ of~$U_E^{0+}$ which are represented by~$\beta$, coincide on~$U_E^{(1/e_{E/F})+}$ and such that~$\chi^{\sigma}=\chi^{-1}$.
\item\label{lem:exist.ii} For each character as in~\ref{lem:exist.i}, and each choice of parity (symplectic/orthogonal) there are at least two extensions to an admissible character~$\chi$ such that~$\rho_\chi$ is irreducible self-dual of that parity.
\end{enumerate}
\end{lem}

We now turn to determinants. For~$E/F$ a finite extension, we define the character~$\varkappa_{E/F}$ of~$W_F$ (which we also identify as a character of~$F^\times$) by 
\[
\varkappa_{E/F} = \det(\Ind_{E/F} \boI_{E^\times}),
\]
where~$\boI_{E^\times}$ denotes the trivial character. According to~\cite[\S29.2]{BH06} this character has order dividing~$2$, and when~$E/F$ is quadratic it is precisely the associated quadratic class field character trivial on~$\N_{E/F}(E^\times)$ (see~\cite[34.3~Proposition]{BH06}). More generally, if~$E\supseteq L\supseteq F$ an intermediate field then, by~\cite[(10.1.3)]{BF83}, we have
\begin{equation}\label{eqn:BF}
\varkappa_{E/F} = \varkappa_{L/F}^{[E:L]} \cdot \varkappa_{E/L}|_{F^\times}
\end{equation}
For~$(E/F,\chi)$ an admissible pair, we write~$\omega_\chi=\det\rho_\chi$ for the determinant of the induced representation. We then have, as a special case of~\cite[29.2~Proposition]{BH06}:

\begin{lem}\label{lem:det}
For~$(E/F,\chi)$ an admissible pair we have
\[
\omega_\chi = \varkappa_{E/F} \cdot \chi|_{F^\times}.
\]
In particular,~$\rho_\chi$ has trivial determinant if and only if~$\chi|_{F^\times}=\varkappa_{E/F}$. 
\end{lem}

We can be more specific about the value of the character~$\varkappa_{E/F}$, from the calculations in~\cite[\S10]{BF83}.

\begin{lem}[{\cite[Propositions~10.1.5,~10.1.6]{BF83}}]\label{lem:BF}
Let~$E/F$ be a tame extension of degree~$N$.
\begin{enumerate}
\item\label{lem:BF.i} If~$E/F$ is unramified then~$\varkappa_{E/F}$ is trivial if~$N$ is odd and unramified nontrivial if~$N$ is even.
\item\label{lem:BF.ii} If~$E/F$ is totally ramified then:
\begin{enumerate}
\item\label{lem:BF.ii.a} if~$N$ is odd then~$\varkappa_{E/F}$ is unramified, and is trivial if and only if the Jacobi symbol~$\binom qN = 1$;
\item\label{lem:BF.ii.b} if~$N$ is even then there is a unique intermediate field~$E\supset L\supseteq F$ with~$E/L$ quadratic, and then~$\varkappa_{E/F}=\varkappa_{E/L}|_{F^\times}$ is a non-trivial ramified character.
\end{enumerate}
\end{enumerate}
\end{lem}

\begin{cor}
Suppose~$E/F$ is tamely ramified of even degree~$2N$. If~$\varkappa_{E/F}$ is trivial then both~$e_{E/F}$ and~$f_{E/F}$ are even. In particular, if~$N$ is odd then~$\varkappa_{E/F}$ is non-trivial.
\end{cor}

\begin{proof}
We prove the contrapositive. Let~$K/F$ be the maximal unramified subextension in~$E$. If~$E/K$ has odd degree then~$K/F$ has even degree so~$q_K$ is a square and~$\varkappa_{E/K}$ is trivial by Lemma~\ref{lem:BF}\ref{lem:BF.ii.a}. But then~\eqref{eqn:BF} implies~$\varkappa_{E/F} = \varkappa_{K/F}$, which is non-trivial unramified by Lemma~\ref{lem:BF}\ref{lem:BF.i}. 

If~$K/F$ has odd degree, so~$E/K$ has even degree, then an element of~$\mu'_F$ is a square if and only if it is a square in~$\mu'_K$. But~$\varkappa_{E/K}$ is the non-trivial quadratic character on~$\mu'_K$, by Lemma~\ref{lem:BF}\ref{lem:BF.ii.b}, so~$\varkappa_{E/F} = \varkappa_{E/K}|_{F^\times}$ is then also a non-trivial ramified quadratic character.
\end{proof}

As an immediate corollary of this, with Proposition~\ref{prop:moy}\ref{prop:moy.i} and Lemma~\ref{lem:det}, we have:

\begin{cor}\label{cor:oddO}
Let~$(E/F,\chi)$ be an admissible pair of degree~$2N$ such that the irreducible representation~$\rho_\chi$ is orthogonal. If~$N$ is odd then~$\omega_\chi$ is non-trivial. 
\end{cor}

On the other hand, we can find totally ramified orthogonal admissible pairs~$(E/F,\chi)$ of degree~$2N$ whose determinant is either of the two ramified quadratic characters, regardless of the parity of~$N$. Indeed, suppose~$\omega$ is a ramified quadratic character of~$F^{\times}$ and choose a uniformizer~$\varpi_F$ such that~$\omega(\varpi_F)=1$. We put~$L=F(\root N\of{\varpi_F})$ and set~$\varpi_L=-\root N\of \varpi_F$, and then take~$E=L(\sqrt{\varpi_L})$. Then~$\varkappa_{E/L}$ is trivial on~$\N_{E/L}(\sqrt{\varpi_L})=-\varpi_L$ so is also trivial on~$\varpi_F=(-\varpi_L)^N$. Since its restriction to~$F^\times$ is a non-trivial ramified quadratic character, by Lemma~\ref{lem:BF}\ref{lem:BF.ii.b}, we must have~$\varkappa_{E/L}|_{F^\times}=\omega$. Now if~$(E/F,\chi)$ is any orthogonal admissible pair, we have~$\omega_\chi=\varkappa_{E/F}=\varkappa_{E/L}|_{F^\times}=\omega$.

\section{Classical Groups}\label{classicalgroupssection}
In this section, we prove the results about classical groups explained in the introduction. We recall first the local Langlands correspondence for split classical groups and observe how the local converse theorem follows from it and the unicity of generic representations in a packet. We assume that~$p$ is odd and, for~$N\ge 1$, write~$G_N$ for one of the groups~$\Sp_{2N}(F),\SO_{2N}(F),\SO_{2N+1}(F)$; in the case of~$\SO_{2N}(F)$ we further assume~$N\ge 2$. We also set~$G_N^+=\OO_{2N}(F)$ when~$G_N=\SO_{2N}(F)$, and~$G_N^+=G_N$ otherwise. We also fix a Whittaker datum $\xi$ for~$G_N$, so that \emph{generic} will always mean relative to this fixed datum; if we wish to specify the datum then we will write~\emph{$\xi$-generic}.

We write~$N^*$ for the dimension of the standard representation of the dual group~$\vG_N(\BC)$ of~$G_N$, so that
\[
N^*=\begin{cases} 2N &\text{ if }G_N=\SO_{2N}(F)\text{ or }\SO_{2N+1}(F), \\
2N+1 &\text{ if }G_N=\Sp_{2N}(F), \end{cases}
\]
and write~$\rhostd:\vG_N(\BC)\hookrightarrow \GL_{N^*}(\BC)$ for the standard representation.

We say that parameters~$\varphi_1,\varphi_2:\WD_F\to\vG_N(\BC)$ are \emph{outer equivalent} if~$\rhostd\circ\varphi_1\cong\rhostd\circ\varphi_2$. Note that this is, a priori, weaker than saying that they are locally conjugate (in the sense of the introduction); in fact, in the case of symplectic and odd orthogonal groups, it is the same as saying that~$\varphi_1,\varphi_2$ are conjugate, while for even orthogonal groups it is the same as saying they are conjugate in the full orthogonal group~$\OO_{2N}(\BC)$.
We write~$\Phi(G_N)$ for the set of (conjugacy classes of) Langlands parameters for~$G_N$, and~$\overline\Phi(G_N)$ for the set of outer equivalence classes. 

Similarly, as in the introduction, we say that irreducible representations~$\pi_1,\pi_2$ of~$G_N$ are \emph{outer conjugate} if there is an element~$g\in G_N^+$ such that~${}^g\pi_1\cong\pi_2$, and we write~$\overline\Irr(G_N)$ for the set outer conjugacy classes. Then, by the main local result of Arthur in~\cite{A15}, we have 
a canonical surjective map on the set of outer conjugacy classes of tempered irreducible representations
\begin{align*}
\oIrrtemp(G_N) &\to \oPhitemp(G_N) \\
[\pi] &\mapsto [\varphi_\pi]
\end{align*}
with finite fibres, where we write~$\varphi_\pi$ for some representative of the class.

On the other hand, if~$\pi$ is a generic irreducible representation of~$G_N$ then Cogdell--Kim--Piatetski-Shapiro--Shahidi constructed its \emph{transfer}~$\Pi$ to~$\GL_{N^*}(F)$ (see~\cite[Proposition~7.2]{CKPSS}), which is a self-dual irreducible representation such that
\[
\gamma(s, \Pi\times \tau, \psi) = \gamma(s, \pi\times \tau, \psi), 
\]
for all irreducible supercuspidal representations~$\tau$ of~$\GL_r(F)$ with~$r \ge 1$. Moreover, by a result of Henniart (see~\cite[Appendix B]{AHKO}), 
this transfer coincides with the Arthur transfer whenever~$\pi$ is supercuspidal generic: that is, the Langlands parameter of~$\GL_{N^*}(F)$ corresponding to~$\Pi$ is~$\rhostd\circ\varphi_\pi$ (which is independent of the choice of representative~$\varphi_\pi$). It also follows from this that the local Langlands correspondence for~$G_N$ preserves twisted~$\gamma$-factors when~$\pi$ is supercuspidal generic:
\[
\gamma(s, \pi\times \tau, \psi) = \gamma(s,\varphi_\pi\times\tau,\psi),
\]
where the~$\gamma$-factor on the RHS is~$\gamma(s,(\rhostd\circ\varphi_\pi)\otimes\tau,\psi)$.

\begin{lem}\label{lem:conjGNGL}
Let $\pi_1, \pi_2$ be generic supercuspidal representations of~$G_N$, with transfers~$\Pi_1,\Pi_2$ to~$\GL_{N^*}(F)$. If~$\Pi_1\cong\Pi_2$ then~$\pi_1,\pi_2$ are outer conjugate.
\end{lem}

\begin{proof}
Since~$\Pi_1\cong\Pi_2$, the local Langlands correspondence for~$\GL_{N^*}(F)$ implies that~$\rhostd\circ\varphi_{\pi_1}\cong\rhostd\circ\varphi_{\pi_2}$. Thus~$\varphi_{\pi_1},\varphi_{\pi_2}$ are outer equivalent. Now a result of Varma~\cite[Corollary~6.16]{V17}
says that there is a unique outer conjugacy class of generic irreducible representations of~$G_N$ corresponding to~$[\varphi_{\pi_1}]$, so~$\pi_1,\pi_2$ are outer conjugate. 
\end{proof}

Up to outer conjugation, the Local Converse Theorem for~$G_N$ is now straightforward for supercuspidal representations, and follows for all generic irreducible representations by standard techniques.

\begin{thm}\label{uniformproof}
Let $\pi_1, \pi_2$ be generic irreducible representations of~$G_N$.  If 
\[
\gamma(s, \pi_1 \times \tau, \psi) = \gamma(s, \pi_2 \times \tau, \psi)
\]
for all generic irreducible representations $\tau$ of $\GL_r(F)$, for $1\le r\le N$, then~$\pi_1,\pi_2$ are outer conjugate. 
\end{thm}

\begin{rmk} 
\begin{enumerate}
\item We note that our version of the local converse theorem removes the assumption in Zhang~\cite{Z18} and Jo~\cite[Theorem~A]{Jo22} for symplectic groups, and in Hazeltine--Liu~\cite[Theorem~1.2]{HZ23} (for even special orthogonal groups), that the central characters of~$\pi_1,\pi_2$ are equal.
\item Another way of expressing Theorem~\ref{uniformproof} is that the map on irreducible $\xi$-generic representations
\[
\oIrrchigen \to \oPhigen
\]
induced by the local Langlands correspondence is \emph{injective}, where a parameter~$\varphi$ is said to be generic if its adjoint~$L$-function~$L(s,\varphi,\Ad)$ is holomorphic at~$s=1$ (see~\cite[\S2.2]{M24} and~\cite[Proposition~B.1]{GI16}). 
Moreover, as pointed out to us by the referee, it is surjective by Cunningham--Dijols--Fiori--Zhang~\cite[Theorem~0.2(2,3)]{CDFZ25}. One can also compare with Jantzen--Liu~\cite[Theorem~6.20]{JL24}, which proves the surjectivity of the map
\[
\bigcup_{\xi}\oIrrchigen \to \oPhigen,
\]
where the union is taken over all (conjugacy classes of) Whittaker data.
\end{enumerate}
\end{rmk}

\begin{proof}
Firstly, let~$\pi_1,\pi_2$ be generic supercuspidal representations of~$G_N$ satisfying the hypotheses of the theorem and let~$\Pi_1,\Pi_2$ be their transfers to~$\GL_{N^*}(F)$. Then, by the definition of transfer,~$\gamma(s, \Pi_1 \times \tau, \psi) = \gamma(s, \Pi_2 \times \tau, \psi)$, for all generic irreducible representations~$\tau$ of~$\GL_r(F)$ with~$1\le r\le N=\left\lfloor \frac{N^*}2\right\rfloor$, so the Local Converse Theorem for~$\GL_{N^*}(F)$ implies that~$\Pi_1\cong\Pi_2$. The result now follows from Lemma~\ref{lem:conjGNGL}.

The reduction to the supercuspidal case is now done as in Jiang--Soudry for odd special orthogonal groups (see~\cite[Theorem~5.1]{JS03}), Jo for symplectic groups (see~\cite[Proposition~2.10]{Jo22}),  and Hazeltine--Liu for even special orthogonal groups (see~\cite[Theorem~4.3]{HZ23} ), noting that the proofs in the latter two cases only require equality of central characters because their versions of the converse theorem for supercuspidal representations also requires this.
\end{proof}

In the following sections, we will discuss the sharpness of Theorem~\ref{uniformproof} (i.e. whether the condition~$r\le N$ can be improved) and the possibility of refining it to remove the outer conjugation (in the case of special orthogonal groups). For ease of exposition, we introduce the following definition.

\begin{defn}
Let~$\pi_1,\pi_2$ be irreducible generic representations of~$G_N$ or of~$\GL_N(F)$ and let~$m\ge 1$ be an integer. We say that~$\pi_1,\pi_2$ are \emph{$\gamma$-equivalent to level~$m$} if
\[
\gamma(s, \pi_1 \times \tau, \psi) = \gamma(s, \pi_2 \times \tau, \psi)
\]
for all irreducible supercuspidal representations $\tau$ of $\GL_r(F)$, with $1\le r\le m$. We use the same terminology for Langlands parameters also.
\end{defn}

Note again that, by multiplicativity of~$\gamma$-factors, if~$\pi_1,\pi_2$ are $\gamma$-equivalent to level~$m$ then we in fact have~$\gamma(s, \pi_1 \times \tau, \psi) = \gamma(s, \pi_2 \times \tau, \psi)$ for all \emph{generic} irreducible representations~$\tau$ of $\GL_r(F)$, with $1\le r\le m$. Thus the Local Converse Theorem~\ref{uniformproof} says that if generic irreducible representations of~$G_N$ are~$\gamma$-equivalent to level~$N$ then they are outer conjugate. 

\subsection{Sharpness Results}\label{sub:sharp}

Our first sharpness result concerns \emph{non-supercuspidal} representations. For this, recall that we have a map
\[
\iota_{\L}: \GL_{N}(\BC) \hookrightarrow \vG_N(\BC)
\]
by identifying~$\GL_N(\BC)$ with a Siegel Levi subgroup~$\vL(\BC)$ of~$\vG_N(\BC)$. Note that, although this embedding is only important up to conjugacy, there are two non-conjugate choices in the case of~$\SO_{2N}(F)$; we fix one arbitrarily. 

For~$(E/F,\chi)$ an admissible pair of degree~$N$ and~$\rho_\chi=\Ind_{E/F} \chi$, we thus obtain a parameter~$\varphi_\chi=\iota_\L\circ\rho_\chi$ for~$G_N$.  Note that, when we compose this with the standard representation, we get
\[
\rhostd\circ\varphi_\chi \cong \begin{cases} 
\rho_\chi\oplus\rho_{\chi^{-1}} &\text{ if }G_N=\SO_{2N}(F)\text{ or }\SO_{2N+1}(F), \\
\rho_\chi\oplus\boI\oplus \rho_{\chi^{-1}} &\text{ if }G_N=\Sp_{2N}(F). \end{cases}
\]
Note then that, if~$(E'/F,\chi')$ is another admissible pair of degree~$N$ such that~$\varphi_\chi,\varphi_{\chi'}$ define equivalent Langlands parameters for~$G_N$, then either~$\rho_{\chi'}\cong\rho_\chi$ or~$\rho_{\chi'}\cong\rho_{\chi^{-1}}$, so~$(E'/F,\chi')$ is conjugate to either~$(E/F,\chi)$ or~$(E/F,\chi^{-1})$.

\begin{thm}\label{thm:noncusp}
Suppose~$p>N$. Let~$(E/F,\chi),(E/F,\chi')$ be minimal totally ramified admissible pairs of degree~$N$ such that~$\chi,\chi'$ coincide on~$U_E$. Then~$\varphi_\chi,\varphi_{\chi'}$ are~$\gamma$-equivalent to level~$N-1$.
\end{thm}

There are infinitely many minimal totally ramified admissible pairs of degree~$N$ with a fixed restriction to~$U_E$ (by choosing the value on a uniformizer), so also infinitely many up to equivalence and duality, even if we also require the image to be bounded and the induced representation not to be self-dual. Thus the local Langlands correspondence for~$G_N$ immediately gives us:

\begin{cor}\label{cor:noncusp}
Suppose~$p>N$. Then there are generic irreducible tempered representations~$\pi,\pi'$ of~$G_N$, which are~$\gamma$-equivalent to level~$N-1$ but not outer conjugate.
\end{cor}

\begin{proof}[Proof of Theorem~\ref{thm:noncusp}]
By definition, we have
\[
\gamma(s, \varphi_\chi \times \tau, \psi) = \gamma(s, (\rhostd\circ \varphi_{\chi}) \otimes \tau, \psi),
\]
to which we seek to apply Proposition~\ref{prop:basic2}. In the symplectic case, the component~$\boI$ in~$\rhostd\circ \varphi_{\chi}$ is the same for~$\chi$ and~$\chi'$ so, by multiplicativity of~$\gamma$-factors, in all cases we need only prove
\begin{equation}\label{eqn:noncusp}
\gamma(s, (\rho_\chi\oplus\rho_{\chi^{-1}}) \otimes \tau, \psi) = \gamma(s, (\rho_{\chi'}\oplus\rho_{\chi'^{-1}}) \otimes \tau, \psi).
\end{equation}
Condition~\ref{prop:basic2.i} of Proposition~\ref{prop:basic2} is immediate from the hypotheses, while~\ref{prop:basic2.iii} is also immediate as~$\chi\chi^{-1}=\chi'\chi'^{-1}$ are trivial, so we only need to verify~\ref{prop:basic2.ii}. If~$\chi,\chi'$ are represented by~$\beta$ then~$\chi^{-1},\chi'^{-1}$ are represented by~$-\beta$ so the condition is
\[
\chi(\beta)\chi^{-1}(-\beta)=\chi'(\beta)\chi'^{-1}(-\beta).
\]
This simply says~$\chi(-1)=\chi'(-1)$, which is true by hypothesis. Thus Proposition~\ref{prop:basic2} says that~\eqref{eqn:noncusp} holds for all irreducible representations $\tau$ of $W_F$ of dimension less than~$N$, as required.
\end{proof}

Now we turn to the case of supercuspidal representations, for which we assume that~$N\ge 2$. We will define some Langlands parameters~$\varphi$ via their composition with~$\rhostd$; as previously discussed, this only determines~$\varphi$ up to outer equivalence but this will be sufficient for our purposes here. We write~$M=2\lfloor\frac N2\rfloor$, so that~$M$ is the largest even integer no greater than~$N$. We also fix arbitrarily a self-dual admissible pair~$(E_0/F,\chi_0)$ of degree~$2$ such that~$E_0/F$ is unramified and~$\rho_{\chi_0}$ is of the same parity as~$\vG_N$ (which imposes a constraint on~$\chi_0|_{F^\times}$ according to Proposition~\ref{prop:moy}).

We will consider a pair of self-dual minimal totally ramified admissible pairs $(E_1/F,\chi_1),(E_2/F,\chi_2)$ of degree~$2M$ which induce to give self-dual irreducible representations~$\rho_{\chi_1},\rho_{\chi_2}$ of the same parity as~$\vG_N$ such that
\[
\omega_{\chi_1}\omega_{\chi_2} = \begin{cases} 
\boI &\text{ if~$N$ is even},\\
\omega_{\chi_0} &\text{ if~$N$ is odd},
\end{cases}
\]
where we recall that~$\omega_\chi=\det\rho_\chi$, for~$(E/F,\chi)$ an admissible quasi-character. Note that this condition is vacuous when~$\vG_N$ is symplectic, and we can always achieve it when~$\vG_M$ is orthogonal by the remarks after Corollary~\ref{cor:oddO}. 

Associated to this pair, we define a Langlands parameter~$\varphi_{\chi_1,\chi_2}$ for~$G_N$ (up to outer equivalence) by
\[
\rhostd\circ \varphi_{\chi_1,\chi_2} = \begin{cases} 
\rho_{\chi_1}\oplus\rho_{\chi_2} &\text{ if~$G_N$ is orthogonal and~$N$ is even},\\
\rho_{\chi_0}\oplus\rho_{\chi_1}\oplus\rho_{\chi_2} &\text{ if~$G_N$ is orthogonal and~$N$ is odd},\\
\boI\oplus\rho_{\chi_1}\oplus\rho_{\chi_2} &\text{ if~$G_N$ is symplectic and~$N$ is even},\\
\boI\oplus\rho_{\chi_0}\oplus\rho_{\chi_1}\oplus\rho_{\chi_2} &\text{ if~$G_N$ is symplectic and~$N$ is odd}.
\end{cases}
\]
Recall also from Lemma~\ref{lem:exist} that there are self-dual minimal totally ramified admissible pairs $(E_1/F,\chi'_1),(E_2/F,\chi'_2)$ of the same parity as~$\vG_N$ such that~$\chi_i,\chi'_i$ differ only (by a sign) on a uniformizer. Notice also that, if~$\beta_i$ represents~$\chi_i$, then~$\chi_i(\beta_i)=-\chi'_i(\beta_i)$, since~$\val(\beta_i)$ has least denominator~$2M$ so is an odd power of a uniformizer. Then we have:

\begin{thm}\label{thm:cusp}
Suppose~$p>N$ and set~$M=2\lfloor\frac N2\rfloor$. With the notation as above, suppose that~$(E_1/F,\chi_1)$ is equivalent to neither~$(E_2/F,\chi_2)$ nor~$(E_2/F,\chi_2')$. Then
\begin{enumerate}
\item\label{thm:cusp.i} $\varphi_{\chi_1,\chi_2}$ and~$\varphi_{\chi'_1,\chi'_2}$ are~$\gamma$-equivalent to level~$M-1$;
\item\label{thm:cusp.ii} $\rhostd\circ\varphi_{\chi_1,\chi_2}\not\cong\rhostd\circ\varphi_{\chi'_1,\chi'_2}$.
\end{enumerate}
\end{thm}

Notice that we can easily find examples such that~$(E_1/F,\chi_1)$ is equivalent to neither~$(E_2/F,\chi_2)$ nor~$(E_2/F,\chi_2')$, for example by requiring that~$\chi_1,\chi_2$ have different depths. Again, the local Langlands correspondence for~$G_N$ immediately gives us:

\begin{cor}\label{cor:cusp}
Suppose~$p>N$ and set~$M=2\lfloor\frac N2\rfloor$. Then there are generic supercuspidal representations~$\pi,\pi'$ of~$G_N$, which are~$\gamma$-equivalent to level~$M-1$ but not outer conjugate.
\end{cor}

\begin{proof}[Proof of Theorem~\ref{thm:cusp}]
By multiplicativity, for~\ref{thm:cusp.i} we need to prove
\[
\gamma(s, (\rho_{\chi_1}\oplus\rho_{\chi_2}) \otimes \tau, \psi) = \gamma(s, (\rho_{\chi_1'}\oplus\rho_{\chi'_2}) \otimes \tau, \psi).
\]
But this follows exactly from Proposition~\ref{prop:basic2}, since the conditions there are clearly satisfied. 

For~\ref{thm:cusp.ii}, we just need to verify that~$\rho_{\chi_1}$ is inequivalent to~$\rho_{\chi'_1}$ (since it is inequivalent to~$\rho_{\chi'_2}$ by hypothesis); that is, we need to check that the admissible pairs~$(E_1/F,\chi_1)$ and~$(E_1/F,\chi'_1)$ are inequivalent. But the characters~$\chi_1,\chi'_1$ agree, and do not factor through any norm, on restriction to~$U_{E_1}^{0+}$, so the only~$F$-automorphism of~$E_1$ which conjugates this restriction is the identity. But~$\chi_1\ne\chi_2$ so we are done.
\end{proof} 

\subsection{Self-dual supercuspidal representations of~$\GL_{2N}(F)$}

When~$N$ is odd, Theorem~\ref{thm:cusp} does not quite close the gap with Theorem~\ref{uniformproof}; that is, it may be that twisting only by representations of dimension up to~$N-1$ \emph{is sufficient} to determine the Langlands parameter of a supercuspidal representation. This is very closely related to the question of whether, for self-dual supercuspidal representations of~$\GL_{2N}(F)$, twisting only by representations of~$\GL_r(F)$ with~$r\le N-1$ is sufficient to distinguish them. 

Let us explain. Write a supercuspidal Langlands parameter $\varphi$ for $G_N$ as a direct sum $\varphi_1 \oplus \cdots \oplus \varphi_n$ of irreducible, self-dual representations of $W_F$.  If $N>1$ is odd, then none of the $\varphi_i$ can be $N$-dimensional, since there is no self-dual irreducible representation of $W_F$ of odd dimension when $p \neq 2$.  Thus, as long as $\varphi$ is not irreducible, twisting by representations of dimensions $1, ..., N-1$ is all that is needed in order to distinguish a parameter via twisted gamma factors.  In searching for a sharpness example to try to show that $N$ is sharp, therefore, we were forced to consider irreducible parameters of $G_N$.  We could not find such a sharpness example, at least when the extension $E/F$ is totally ramified.  Only on this basis, we felt that we could make the following conjecture:

\begin{conj}\label{conj:better}
Suppose~$p>N$ and suppose that $N$ is odd. Let~$\pi,\pi'$ be self-dual supercuspidal representations of~$\GL_{2N}(F)$ of the same parity which are~$\gamma$-equivalent to level~$N-1$.
Then~$\pi\cong\pi'$.
\end{conj}

\begin{rmk}
We have no evidence as to whether the assumption that $p>N$ or the assumption that $N$ be odd are needed.
\end{rmk}

The following theorem shows that Conjecture \ref{conj:better} is almost best possible: the bound can be improved at most to~$N-2$. We have no evidence as to whether this improvement in bound is possible. The theorem also shows that, in the case that~$N$ is even, the same statement as in Conjecture \ref{conj:better} would be the best possible.

\begin{thm}\label{thm:better}
Suppose~$p>N$, set~$M=2\lfloor\frac N2\rfloor$, and fix a parity. There exist inequivalent self-dual supercuspidal representations~$\pi,\pi'$ of~$\GL_{2N}(F)$ of this parity which are~$\gamma$-equivalent to level~$M-2$.
\end{thm}

\begin{proof}
As in other results, by the local Langlands correspondence we translate to a question about self-dual irreducible representations of~$W_F$. Let~$L/F$ be a totally ramified extension of degree~$N$, and~$E/L$ a ramified quadratic extension with galois involution~$\sigma$. We set
\[
m=\begin{cases} N+1 &\text{ if~$N$ is even},\\
\frac{3N+(-1)^{(N+1)/2}}{2} &\text{ if~$N$ is odd}.
\end{cases}
\]
Note that~$m$ is odd and~$d=\frac m{2N}$ has least denominator~$2N$. Now an easy exercise in congruences shows that the least integer~$r$ such that~$rm\equiv \pm 1\pmod{2N}$ is~$r=M-1$, so that Lemma~\ref{lem:tbeta} implies that~$t_\beta(M-1)>\frac 1{2N}$.

Now let~$\beta\in E$ be an element of valuation~$-d$ such that~$\sigma(\beta)=-\beta$, and let~$(E/F,\chi)$ and~$(E/F,\chi')$ be self-dual admissible characters of the required parity which are represented by~$\beta$, coincide on~$U_E^{(1/2N)+}$, and such that~$\chi(\beta)=\chi'(\beta)$; these exist by Lemma~\ref{lem:exist}, noting that since~$\beta^2\in L^\times$ we must have~$\chi(\beta^2)=\chi'(\beta^2)$ and both choices of square root are permitted.
Then Proposition~\ref{prop:basic2} gives the result. 
\end{proof}


\subsection{A (conjectural) improved bound for classical groups when~$N$ is odd}
When~$N$ is odd, Conjecture~\ref{conj:better} implies an improved bound for the Local Converse Theorem for the classical group~$G_N$, which would then mean that our sharpness result in Corollary~\ref{cor:cusp} is optimal.

\begin{thm}\label{thm:Gbetter}
Suppose~$N$ is odd and $p>N$. Suppose also that Conjecture~\ref{conj:better} is true for~$\GL_{2N}(F)$. Let~$\pi_1,\pi_2$ be irreducible generic supercuspidal representations of~$G_N$ which are~$\gamma$-equivalent to level~$N-1$.
Then~$\pi_1, \pi_2$ are outer conjugate.  
\end{thm}

The proof of this relies on the following (unconditional) lemma.

\begin{lem}\label{lem:nonmaxbetter}
Suppose~$N$ is odd and $p>N$. Let~$\pi_1$ be an irreducible generic supercuspidal representation of~$G_N$ such that largest irreducible component of the restriction to~$W_F$ of~$\rhostd\circ\varphi_{\pi_1}$ has dimension strictly less than~$2N$. Let~$\pi_2$ be an irreducible generic supercuspidal representation of~$G_N$ such that~$\pi_1,\pi_2$ are~$\gamma$-equivalent to level~$N-1$.
Then~$\pi_1, \pi_2$ are outer conjugate.  
\end{lem}

\begin{proof}
Let~$\pi_1,\pi_2$ be irreducible generic supercuspidal representations of~$G_N$ satisfying the hypotheses, so that
\[
\gamma(s, \pi_1 \times \tau, \psi) = \gamma(s, \pi_2 \times \tau, \psi),
\]
for all irreducible supercuspidal representations~$\tau$ of~$\GL_r$ with~$r = 1, \ldots, N-1$. We abbreviate~$\varphi_i^{\GL}=\rhostd\circ\varphi_{\pi_i}$ and let~$\Pi_1,\Pi_2$ be the transfers to~$\GL_{N^*}(F)$. Then we can write
\[
\varphi_i^{\GL} \cong \bigoplus_{j=1}^{t_i} \varphi_i^{(j)}\otimes\left(\st_{a_i^{(j)}}\oplus\st_{a_i^{(j)}-2}\oplus\cdots\right),
\] 
where~$\st_a$ is the unique~$a$-dimensional irreducible representation of~$\SL_2(\BC)$ and the~$\varphi_i^{(j)}$ are distinct self-dual irreducible representations of the Weil group~$W_F$. Then, writing~$\varphi_\tau$ for the Langlands parameter corresponding to~$\tau$, we have
\begin{equation}\label{eqn:transfer}
\gamma(s,\Pi_i\times \tau,\psi) = \gamma(s,\varphi_i^{\GL}\otimes \varphi_\tau,\psi) = \prod_{j=1}^{t_i} \prod_{k=\frac{1-a_i^{(j)}}2}^{\frac{a_i^{(j)}-1}2} \gamma(s+k,\varphi_i^{(j)}\otimes\varphi_\tau,\psi).
\end{equation}
By Lemma~\ref{lem:conjGNGL}, we need only show that~$\varphi_1^{\GL}$ is equivalent to~$\varphi_2^{\GL}$. For contradiction, we suppose not. Then there are an irreducible self-dual representation~$\varphi$ of~$W_F$ and natural number~$a$ such that~$\varphi\otimes\st_a$ appears in~$\varphi_1^{\GL}$ but not~$\varphi_2^{\GL}$, and we can choose the maximal such~$a$. 

Suppose first that there is such~$\varphi$ with~$\dim(\varphi)\le N^*/2$, so that~$\dim(\varphi)\le N$. As~$N\ge 3$ is odd while~$\dim(\varphi)$ is~$1$ or even, we see that in fact~$\dim(\varphi)\le N-1$. But then, writing~$\pi_\varphi$ for the corresponding (self-dual) irreducible representation a general linear group,~\cite[Proposition~2.9]{JNS15} implies that~$\gamma(s,\Pi_1\times \pi_\varphi^\vee,\psi)$ has a real pole at~$s=\frac{a+1}2$, while~$\gamma(s,\Pi_2\times \pi_\varphi^\vee,\psi)$ does not, a contradiction.

Thus the only such~$\varphi$ must have~$\dim(\varphi)> N^*/2$. Then~$\varphi\otimes\st_1$ is in fact the \emph{unique} component of~$\varphi_1^{\GL}$ which does not appear in~$\varphi_2^{\GL}$ (since every other such component would have dimension less than~$N^*/2$). Moreover, there is also a unique component~$\varphi'\otimes\st_1$ of~$\varphi_2^{\GL}$ which does not appear in~$\varphi_1^{\GL}$ (again, if there were two then at least one would have dimension less than~$N^*/2$, which is impossible by the previous paragraph). But then, since all other terms in~\eqref{eqn:transfer} are identical for~$i=1,2$, we have that~$\dim(\varphi)=\dim(\varphi')< 2N$ (by hypothesis) and they are~$\gamma$-equivalent to level~$N-1$. Then Theorem~\ref{uniformproof} implies that~$\varphi,\varphi'$ are equivalent, which is absurd.
\end{proof}

\begin{proof}[Proof of Theorem~\ref{thm:Gbetter}]
By Lemma~\ref{lem:nonmaxbetter}, we need only consider the case where~$\rhostd\circ\varphi_{\pi_1}$ and~$\rhostd\circ\varphi_{\pi_2}$ are irreducible of dimension~$2N$ or the sum of a quadratic character and an irreducible representation of dimension~$2N$. As in the proof of Lemma~\ref{lem:nonmaxbetter}, in the latter case we see that the quadratic character is the same so, by multiplicativity of~$\gamma$-factors, we reduce to the first case. But then these are the Langlands parameters of a pair of self-dual supercuspidal representations of~$\GL_{2N}$ of the same parity, so Conjecture~\ref{conj:better} implies they are equivalent. Thus~$\varphi_{\pi_1},\varphi_{\pi_2}$ are outer equivalent, so~$\pi_1,\pi_2$ are outer conjugate.
\end{proof}

In fact, we can prove this improved Local Converse Theorem in the case of even special orthogonal groups.

\begin{thm}\label{thm:SObetter}
Suppose~$N$ is odd and $p>N$. Let~$\pi_1,\pi_2$ be irreducible generic supercuspidal representations of~$\SO_{2N}(F)$ which are~$\gamma$-equivalent to level~$N-1$. Then~$\pi_1, \pi_2$ are outer conjugate.  
\end{thm}

\begin{proof}
By Lemma~\ref{lem:nonmaxbetter}, we need only consider the case where~$\rhostd\circ\varphi_{\pi_1}$ is irreducible. But, by Corollary~\ref{cor:oddO}, there are no irreducible orthogonal representations of~$W_F$ of dimension~$2N$ and determinant~$1$ when~$N$ is odd, so there is nothing to prove.
\end{proof}

\subsection{Generic $\WD_F$-unacceptability of~$\SO_6$}
In the case of even special orthogonal groups, twisted standard~$\gamma$-factors cannot distinguish between Langlands parameters~$\varphi,\varphi'$ which are outer equivalent. One might hope that we can use different algebraic representations of the dual group to help distinguish them. The representation ring of~$\SO_{2N}(\BC)$ is generated by the fundamental representations~$\Wedge^1,\cdots,\Wedge^{N-1},\Wedge^N_+,\Wedge^N_-$, where~$\Wedge^r$ denotes the~$r^{\text{th}}$ exterior power of~$\rhostd$ and~$\Wedge^N_{\pm}$ are described in~\cite{YZ24} (see also below in the case~$N=3$). We cannot hope to use the exterior powers to distinguish~$\varphi$ from~$\varphi'$ since
\[
\Wedge^i\circ \varphi = \Wedge^i\circ\rhostd\circ\varphi 
\]
and~$\rhostd\circ\varphi\cong\rhostd\circ\varphi'$. Thus it is only using the representations~$\Wedge^N_+,\Wedge^N_-$ that we can hope to distinguish them. Since~$\Wedge^N_+$ is obtained by composing $\Wedge^N_-$ with the outer automorphism, we get that $\Wedge^N_+ \circ \varphi' \cong \Wedge^N_- \circ \varphi$; thus we can use these to distinguish~$\varphi$ from~$\varphi'$ precisely when~$\Wedge^N_+\circ\varphi\not\cong\Wedge^N_-\circ\varphi$.

In the case of~$\SO_4(F)$ it is shown in~\cite{YZ24} that this is always the case, so that twisted~$\gamma$-factors can be used to distinguish outer equivalent parameters. However, this seems unlikely to be true in general as~$\SO_{2N}(\BC)$ is not acceptable for~$N\ge 3$. 
The following example shows that, for~$\SO_6(F)$ when~$q\equiv 3\pmod 4$, there is indeed a parameter~$\varphi$ which is not conjugate to its outer twist such that~$\Wedge^N_+\circ\varphi \cong\Wedge^N_-\circ\varphi$.  We note that in~\cite[Proposition~7.3]{M24}, an example for $\SO_{2N}(F)$ is given when $q\equiv 1\pmod 4$; one difference is that our example is a discrete parameter (indeed the representations in the corresponding~$L$-packet are all supercuspidal) while that in~\cite{M24} is not.

\begin{exmp}\label{so6}
Let~$L/F$ be the biquadratic extension of~$F$, with~$E_1,E_2,E_3$ the quadratic subfields and~$E_1/F$ unramified. Let~$\zeta\in\mu'_{E_1}$ be a generator, so that~$\zeta$ has order~$q^2-1$, and let~$\varpi_F$ be a fixed uniformizer of~$F$. Define a complex character~$\chi_1$ of~$E_1^\times$ by 
\[
\chi_1(\varpi_F^n \zeta^k u_1) = \mathrm{i}^k, \quad\text{ for }u_1\in U_{E_1}^{0+},
\]
where~$\mathrm{i}$ is a fixed complex square root of~$-1$. Since~$q\equiv 3\pmod 4$, this is an orthogonal admissible character of order~$4$. We pick orthogonal admissible ramified characters~$\chi_2,\chi_3$, of~$E_2,E_3$ respectively, of distinct positive depths. Define~$\varphi$ up to outer equivalence by
\[
\rhostd\circ \varphi = \textstyle\bigoplus\limits_{i=1}^3 \rho_{\chi_i}.
\]
This does define a parameter into~$\SO_6(\BC)$ since~$\det\rho_{\chi_i}=\omega_{E_i/F}$, the local class field character, and the product of the three class field characters is trivial. Then we claim that~$\Wedge^3_+\circ\varphi\cong\Wedge^3_-\circ\varphi$, and that~$\varphi$ is not \shaun{equivalent} to its own outer conjugate.
\end{exmp}

\begin{proof}
We first explain why our parameter~$\varphi$ is not \shaun{equivalent} to its own outer conjugate. \shaun{Suppose~$g\in\OO(6,\BC)$ is such that there exists~$h\in\SO(6,\BC)$ such that~${}^g\varphi={}^h\varphi$. Then~$\rhostd(h^{-1}g)$ centralizes~$\rhostd\circ\varphi= \rho_{\chi_1}\oplus\rho_{\chi_2}\oplus\rho_{\chi_3}$ so, by Schur's Lemma,~$\rhostd(h^{-1}g)$ lies in $\{(\pm \mathrm{I}_2, \pm \mathrm{I}_2, \pm \mathrm{I}_2)\}$, where $\mathrm{I}_2$ is the~$2$-by-$2$ identity matrix.  But the determinant of such a matrix is equal to one, and thus $g$ also lies in~$\SO(6,\mathbb{C})$.  In particular, we conclude that~$\varphi$ is not equivalent to its own outer conjugate.}

We now compute~$\Wedge^3\circ\varphi$, which is the sum of the two~$\Wedge^3_{\pm}\circ\varphi$. Note that we are really computing~$\rhostd\circ\Wedge^3\circ\varphi$, but we will omit the~$\rhostd$ to simplify notation. We abbreviate~$\rho_i=\rho_{\chi_i}$ and~$\omega_i=\omega_{E_i/F}$. Since~$\Wedge^2\rho_i=\det\rho_i=\omega_i$, we have
\[
\Wedge^3\circ\varphi \cong \rho_1\otimes\rho_2\otimes\rho_3\oplus \textstyle\bigoplus\limits_{i\ne j} \rho_i\otimes\omega_j.
\]
Moreover, applying Lemma~\ref{lem1} twice and using that the~$\chi_i$ are self-dual, we have
\[
\rho_1\otimes\rho_2\otimes\rho_3\cong \Ind_{L/F} \left(\textstyle\prod\limits_{i=1}^3 \chi_{i,L}\right) \oplus \Ind_{L/F} \left(\textstyle\prod\limits_{i=1}^3 \chi^{-1}_{i,L}\right),
\]
where we recall that~$\chi_{i,L}=\chi_i\circ\N_{L/E_i}$. Moreover, in our situation, these two~$4$-dimensional representations are irreducible and isomorphic. Indeed,~$\chi_{1,L}$ is a quadratic character because~$\N_{L/E_1}(L^\times)=\varpi_F^{\BZ}\langle\zeta^2\rangle U_{E_1}^{0+}$, while the characters~$\chi_{2,L}$, $\chi_{3,L}$ and~$\chi_{2,L}\chi_{3,L}$ are all non-quadratic (since they have positive depth), so each~$4$-dimensional piece restricts to~$W_L$ to a sum of distinct characters and is thus irreducible; but both restrict to the sum of the same four characters, so they are isomorphic.

Moreover, for~$i\ne j$, we have~$\omega_j\circ\N_{E_i/F}=\omega_{L/E_i}$ so
\[
\rho_i\otimes\omega_j \cong \Ind_{E_i/F} (\chi_i\omega_{L/E_i}),
\]
which is independent of~$j$. Thus we have a decomposition
\[
\Wedge^3\circ\varphi \cong \left(\Ind_{L/F} \left(\textstyle\prod\limits_{i=1}^3 \chi_{i,L}\right) \oplus \Ind_{L/F} \textstyle\bigoplus\limits_{i=1}^3\Ind_{E_i/F} (\chi_i\omega_{L/E_i})\right)^{\oplus 2},
\]
in which each representation appearing is irreducible and is determined by the characters appearing in its restriction to~$W_L$. Thus, to verify that~$\Wedge^3_+\circ\varphi\cong\Wedge^3_-\circ\varphi$, we need only check that they have the same restrictions to~$W_L$.

We now use the general description of~$\Wedge^n_{\pm}$ from~\cite{YZ24} to make these explicit. We write~$(W,\langle\cdot,\cdot\rangle)$ for the~$6$-dimensional complex orthogonal space, with Witt basis~$e_1,e_2, e_3, e_{-3}, e_{-2}, e_{-1}$, so that
\[
\langle e_i,e_j\rangle = 
\begin{cases} 
1 &\text{ if~$j=-i$}, \\ 
0 &\text{ otherwise.} 
\end{cases}
\]
Then~$\SO_6(\mathbb{C})$ is (identified with) the group of automorphisms of~$W$ which preserve the symmetric form~$\langle\cdot,\cdot\rangle$ and have determinant~$1$. We define an involution~$\tau$ on~$\Wedge^3 W$ as follows: for a simple tensor~$x=e_i\wedge e_j\wedge e_k$, we let~$\tau(x)$ be the unique simple tensor~$y$ such that
\[
e_{-i}\wedge e_{-j}\wedge e_{-k} \wedge y = e_1 \wedge e_2 \wedge e_3 \wedge e_{-3} \wedge e_{-2} \wedge e_{-1}. 
\]
Then~$\Wedge^3_\pm(W)$ is the~$(\pm 1)$-eigenspace of~$\tau$ on~$W$ and we can write down bases for them: for~$(i,j,k)$ any cyclic permutation of~$(1,2,3)$, we write
\begin{align*}
u_i^{\pm} &= e_{\pm i}\wedge e_{\mp j}\wedge e_{\mp k},  \\
v_i^{\pm} &= e_i \wedge e_j\wedge e_{-j} \pm e_i\wedge e_k\wedge e_{-k}, \\
v_{-i}^{\pm} & = e_{-i} \wedge e_j\wedge e_{-j} \mp e_{-i}\wedge e_k\wedge e_{-k};
\end{align*}
then~$\Wedge^3_\pm(W)$ 
 has basis
\[
\left\{e_{\pm 1}\wedge e_{\pm2}\wedge e_{\pm 3}, u_i^\pm,v_i^{\pm},v_{-i}^\pm : i=1,2,3\right\}.
\]
Now we realise our parameter~$\varphi$ with~$W_{E_i}$ acting on~$e_i$ via the character~$\chi_i$ (so on~$e_{-i}$ via the character~$\chi_i^{-1}$), and check the action of~$W_L$: since~$\chi_{1,L}=\chi_{1,L}^{-1}$, the characters appearing in~$\Res_{W_L}^{W_F}\Wedge^3_+\circ\varphi$ are the same as those in~$\Res_{W_L}^{W_F}\Wedge^3_-\circ\varphi$, and we are done.
%

Thus, we have produced a parameter $\varphi$ for $\SO_6(F)$ that is not isomorphic to its own outer conjugate, and which satisfies $\Wedge^3_+ \circ \varphi = \Wedge^3_- \circ \varphi$, completing the proof.
\end{proof}

It is now not hard to use this example to deduce the following:

\begin{cor}\label{cor:SO2N} 
The group~$\SO_{2N}(\BC)$ is generically~$\WD_F$-unacceptable.
\end{cor}

We thank Nadir Matringe for pointing out how this should be possible, using the ideas in Yu~\cite{Yu22}. We also note that Matringe proved this (in a slightly different way) when~$q\equiv 1\pmod 4$ in~\cite[Proposition~7.3]{M24}.

\begin{proof}
By~\cite[Proposition~7.3]{M24}, we need only treat the case~$q\equiv 3\pmod 4$ so we let~$\varphi:W_F \to\SO_6(\BC)$ be the Langlands parameter constructed in Example~\ref{so6}. Then there is a finite Galois extension~$K/F$ such that~$\varphi$ is trivial on~$W_K$, so~$\varphi$ factors through the finite group~$\Gamma:=\Gal(K/F)$. Now let~$L/F$ be an unramified extension of degree~$M>2N$ such that~$M$ is prime to the degree of~$K/F$, and set~$E=KL$. Then, since the degrees of the extensions are coprime, we have
\[ 
\Gal(E/F) \simeq \Gal(K/F)\times \Gal(L/F) \simeq \Gamma \times \BZ/M\BZ.
 \]
Thus~$\Gamma\times\BZ/M\BZ$ is isomorphic to the quotient~$W_F/W_E$ of~$W_F$ and we can now apply the method in~\cite[Lemma~2.3]{Yu22} to construct a parameter with image~$\SO_{2N}(\BC)$ as follows. We pick a primitive~$M^\text{th}$ root of unity~$\zeta$ in~$\BC$ and define~$\varphi'$ to be the inflation to~$W_F$ of the map
\[
\Gamma \times \BZ/M\BZ \to \SO_{2N}(\BC), \quad (\gamma, a) \mapsto \diag( \zeta^a, \ldots ,\zeta^{(N-3)a}, \varphi(\gamma), \zeta^{(3-N)a},\ldots, \zeta^{-a}),
\]
 for~$\gamma\in\Gamma$ and~$a\in\BZ/M\BZ$.

Now take~$\varphi_1=\varphi$ and let~$\varphi_2$ be its outer conjugate, from which we construct morphisms~$\varphi_1',\varphi_2'$ (with the same choice~$\zeta$).Then~$\varphi’_1,\varphi’_2$ are clearly also outer equivalent. If they were conjugate by some~$g\in\SO_{2N}(\BC)$ then we would have
\[
\varphi'_1(\gamma,a)=\Ad(g)\circ \varphi'_2(\gamma,a), \quad\text{for all }\gamma\in\Gamma, a\in\BZ/M\BZ.
\]
By considering the element~$(\gamma,a)=(1,1)$, we see that~$g$ is in the centralizer
\[ 
Z_{\SO_{2N}(\BC)}( \diag( \zeta, \ldots ,\zeta^{N-3}, \Id_6 , \zeta^{3-N},\ldots, \zeta^{-1}) \simeq \GL_1(\BC) \times \cdots \times \GL_1(\BC) \times \SO_6(\BC).
\]
But then taking the elements~$(\gamma,0)$, for all ~$\gamma\in\Gamma$, we would get that~$\varphi_1,\varphi_2$ are conjugate in~$\SO_6(\BC)$, a contradiction. 

\end{proof}

\begin{rmk}\label{rmk:SO2odd}
As remarked above, one difference between our example for~$\SO_6(\BC)$  (when~$q\equiv 3\pmod 4$) and the example of Matringe (when~$q\equiv 1\pmod 4$) is that ours is a discrete parameter whilst his is not. The parameter constructed in Corollary~\ref{cor:SO2N} is also not discrete.

However, the only thing we have really used in our~$\SO_6(\BC)$ example is the fact that the character~$\chi_{1,L}=\chi_1\circ\N_{L/E_1}$ is quadratic. By a similar (but notationally more complicated!) method, one can find a parameter~$\varphi$ for~$\SO_{2(2N+1)}(F)$ such that~$\Wedge^{2N+1}_+\circ\varphi\cong\Wedge^{2N+1}_-\circ\varphi$, so that one also sees the non-acceptability of~$\SO_{2(2N+1)}(\BC)$ manifested by a Langlands parameter. One need only add to~$\rho_{\chi_1}$ a sum of (pairwise inequivalent) irreducible~$2$-dimensional orthogonal representations of distinct positive depths induced from ramified quadratic extensions (necessarily an odd number from each of the two quadratic ramified extensions, so that the parameter has determinant~$1$).

Thus it is possible to find non-conjugate discrete parameters which are outer-equivalent for~$\SO_{2N}(F)$ whenever~$N$ is odd and~$q\equiv 3\pmod 4$. We do not give details here, nor investigate the situation for general~$\SO_{2N}(F)$ further, leaving this to future work, or others.
\end{rmk}


\section{The exceptional group~$\G_2$}\label{g2section}

Now we turn to the case of the exceptional group~$\G_2$. Since there is not yet a direct definition of twisted~$\gamma$-factors on the automorphic side, we will follow Gan--Savin~\cite{GS23b}, where they \emph{define} these~$\gamma$-factors to be those of the corresponding Langlands parameter, via their local Langlands correspondence (see~\cite{GS23a}). 

\subsection{The converse theorem for parameters for $\G_2(F)$}
We consider first a Langlands parameter $\varphi : \WD_F \rightarrow \G_2(\mathbb{C})$ for $G_2$ and the standard embedding $\rhostd : \G_2(\mathbb{C}) \hookrightarrow \GL_7(\mathbb{C})$. Then, for $\tau : W_F \rightarrow \GL_r(\mathbb{C})$ an irreducible representation, by definition we have
\[
\gamma(s, \varphi \times \tau, \psi)=\gamma(s, (\rhostd\circ \varphi) \otimes \tau, \psi).
\]
As for classical groups, we make the following definition.

\begin{defn} Let~$\varphi_1,\varphi_2$ be Langlands parameters for~$\G_2(F)$ and let~$m\ge 1$ be an integer. We say that~$\varphi_1,\varphi_2$ are \emph{$\gamma$-equivalent to level~$m$} if
\[
\gamma(s, \varphi_1 \times \tau, \psi) = \gamma(s, \varphi_2 \times \tau, \psi)
\]
for all irreducible irreducible~$r$-dimensional representations $\tau$ of $W_F$, with $1\le r\le m$.
\end{defn}

The following theorem seems to be implicit in~\cite{GS23b}, and see also~\cite[Theorem~8.2(b)]{M24}. 

\begin{thm}\label{g2converse}
Let $\varphi_1, \varphi_2$ be two Langlands parameters for $\G_2(F)$.  If~$\varphi_1,\varphi_2$ are $\gamma$-equivalent to level~$3$, then $\varphi_1 \cong \varphi_2$. 
\end{thm}

We will see in the next subsection how our results for classical groups can be used to show that this is the best possible, at least when~$p>3$.

\begin{proof}
By~\cite[Theorem~8.2(b)]{M24}, the hypotheses imply that~$\rhostd\circ\varphi_1,\rhostd\circ \varphi_2$ are conjugate under the normalizer in~$\GL_7(\mathbb{C})$ of~$\rhostd(\G_2(\mathbb C))$; since this normalizer is~$\mathbb{C}^\times\rhostd(\G_2(\mathbb C))$ (see the proof in \emph{loc. cit.}), we deduce that they are in fact conjugate under~$\rhostd(\G_2(\mathbb C))$, so that the parameters~$\varphi_1,\varphi_2$ are~$\G_2(\mathbb{C})$-conjugate, as required. 
\end{proof}

\begin{rmk} 
As the referee has pointed out to us, Theorem~\ref{g2converse} also follows from the fact that~$\rhostd$ factors through the inclusion~$\G_2(\mathbb C)\hookrightarrow\SO_7(\mathbb C)$: thus we get an injective map~$\Phi(\G_2)\hookrightarrow\Phi(\Sp_6)$ (see~\cite[Lemma~2.3(ii)]{GS23a}) and the result follows from the local converse theorem for~$\Sp_6$.
\end{rmk}

\subsection{On the optimal bound for parameters for $\G_2(F)$}\label{g2parameters}

For the points~$\G(F)$ of a reductive group over~$F$, we write~$\Phi_{\dss}(\G)$ for the set of discrete Langlands parameters for~$\G(F)$. By~\cite[Lemma~2.4(i)]{GS23a}, the inclusion map~$\iota:\SL_3(\mathbb{C}) \to \G_2(\mathbb{C})$ induces a map
\begin{equation}\label{eqn:PGLtoG}
\Phi_{\dss}(\PGL_3) \to \Phi_{\dss}(\G_2)
\end{equation}
whose fibres consist of orbits under the outer automorphism of~$\SL_3(\mathbb{C})$. In particular, for~$\rho : W_F \rightarrow \SL(3,\mathbb{C})$ a parameter for~$\PGL_3(F)$, we have
\[
\rhostd\circ\iota\circ\rho = \rho \oplus \boI \oplus \rho^{\vee}.
\]
We also recall that an admissible pair~$(L/F, \chi)$ gives rise to a parameter for~$\PGL_3(F)$ if and only if the induced representation~$\rho_\chi=\Ind_{L/F}\chi$ has determinant~$1$; equivalently,~$\chi|_{F^{\times}} = \varkappa_{L/F}$, where~$\varkappa_{L/F} = \det(\Ind_{L/F} \boI_{L^\times}).$  

\begin{thm}\label{G2}
Suppose $p > 3$.  There exist inequivalent discrete Langlands parameters for~$\G_2(F)$ which are~$\gamma$-equivalent to level~$2$.
\end{thm}

\begin{proof} 
We make use of our result on symplectic groups in Theorem~\ref{thm:noncusp}. More precisely, suppose we have found minimal totally ramified admissible pairs~$(E/F,\chi_1)$ and~$(E/F,\chi_2)$, with~$E/F$ a cubic extension, such that
\begin{enumerate}
\item $\chi_i|_{F^\times}=\varkappa_{E/F}$; 
\item $\chi_1|_{U_E}=\chi_2|_{U_E}$; and
\item\label{cond.iii} $(E/F,\chi_1)$ is not equivalent to~$(E/F,\chi_2^{\pm 1})$.
\end{enumerate}
Then, writing~$\rho_i=\Ind_{E/F}\chi_i$, we get a parameter~$\varphi_i=\iota\circ\rho_i$ for~$\G_2(F)$ and a parameter~$\phi_i=\iota_L\circ\rho_i$ for~$\Sp_6(F)$ (in the notation of~\S\ref{sub:sharp}) such that~$\rhostd\circ\varphi_i \simeq \rhostd\circ\phi_i$ (where we are abusing notation by using~$\rhostd$ denote the standard representation of both~$\G_2(\mathbb{C})$ and~$\SO_7(\mathbb{C})$). Then Theorem~\ref{thm:noncusp} implies that~$\phi_1,\phi_2$ are~$\gamma$-equivalent to level~$2$, so the same is true of~$\varphi_1,\varphi_2$. Moreover, the final condition~\ref{cond.iii} ensures that~$\varphi_1,\varphi_2$ are inequivalent as parameters for~$\G_2(F)$.

Now we explain how to find admissible pairs satisfying these conditions. We take~$E/F$ to be totally ramified of degree~$3$, with uniformizer~$\varpi_E$ such that~$\varpi_E^3=\varpi_F$ lies in~$F$. Since~$p>3$, this is a tamely ramified extension and, by Lemma~\ref{lem:BF}, the character~$\varkappa_{E/F}$ is unramified of order dividing~$2$. 
 We choose any non-trivial character of~$U_E^{0+}$ which is trivial on~$U_E^{(1/3)+}$ (for example, represented by~$\varpi_E^{-1}$); since~$U_F^1\subset U_E^{(1/3)+}$, this is trivial on~$U_F^1$ so we can extend it to a character~$\chi$ of~$F^\times U_E^{0+}$ by requiring~$\chi|_{F^\times}=\varkappa_{E/F}$. 

Finally, we extend~$\chi$ to two distinct characters~$\chi_1,\chi_2$ of~$E^\times$. Any such extension must send~$\varpi_E$ to a cube root of~$\varkappa_{E/F}(\varpi_F)=\pm 1$. So we set~$\chi_1(\varpi_E)=\varkappa_{E/F}(\varpi_F)$ and~$\chi_2(\varpi_E)=\zeta \varkappa_{E/F}(\varpi_F)$, where~$\zeta\in\mathbb{C}^\times$ is a primitive cube root of unity. 

It remains only to check the final condition~\ref{cond.iii}. But if~$\sigma\in\Aut(E/F)$ then~$\sigma(\varpi_E)$ differs from~$\varpi_E$ by a cube root of unity in~$U_F$, on which~$\varkappa_{E/F}$ is trivial; hence we deduce~${}^\sigma\chi_1(\varpi_E)=\chi_1(\varpi_E) \ne \chi_2(\varpi_E)^{\pm 1}$.
\end{proof}

\begin{rmk}\label{rem:G2}
The representations~$\rho_1,\rho_2$ in the proof of Theorem~\ref{G2} have the further property that~$\rho_1\otimes\rho_1^\vee \simeq \rho_2\otimes\rho_2^\vee$. Indeed, if~$E/F$ is Galois, with automorphism group generated by~$\sigma$, then
\[
\rho_i\otimes\rho_i^\vee \simeq {\ts \bigoplus\limits_{j=0}^2} \Ind_{E/F} {}^{\sigma^j}\chi_i\chi_i^{-1};
\]
but, for fixed~$j$, the characters~${}^{\sigma^j}\chi_i\chi_i^{-1}$ certainly coincide on~$U_E$ by construction, while we have seen that~${}^\sigma\chi_i(\varpi_E)=\chi_i(\varpi_E)$ so that~${}^{\sigma^j}\chi_i\chi_i^{-1}$ are both trivial on~$\varpi_E$. On the other hand, if~$E/F$ is not Galois then, writing~$L/F$ for the unramified quadratic extension and~$EL$ for the compositum of~$E$ and~$L$, we have
\[
\rho_i\otimes\rho_i^\vee \simeq \Ind_{E/F} \boI \oplus \Ind_{EL/F} \left({}^\sigma\chi_{i,EL}\chi_{i,EL}^{-1}\right),
\]
where~$\chi_{i,EL}=\chi_i\circ\N_{EL/E}$ and~${}^\sigma\chi_{i,EL}={}^\sigma\chi_i\circ\N_{EL/\sigma(E)}$. Again, the characters~${}^\sigma\chi_{i,EL}\chi_{i,EL}^{-1}$ coincide on~$U_{EL}$ by construction, and it is easy to check they are trivial on the uniformizer~$\varpi_E$ of~$EL$.

Now 
\[ 
\Wedge^2\circ \rhostd\circ\varphi_i \simeq  \Wedge^2\left( \rho_i\oplus\boI\oplus\rho_i^\vee\right) 
\simeq \rho_i^{\oplus 2} \oplus \left(\rho_i\otimes\rho_i^\vee\right) \oplus (\rho_i^\vee)^{\oplus 2},
\] 
where we have used that~$\wedge^2\rho_i\simeq\rho_i^\vee$, since~$\rho_i$ is~$3$-dimensional and has trivial determinant. In particular, from Theorem~\ref{G2} and the isomorphism~$\rho_1\otimes\rho_1^\vee \simeq \rho_2\otimes\rho_2^\vee$, we deduce that
\[
\gamma(s,(\Wedge^2\circ\rhostd\circ\varphi_1)\otimes\tau,\psi) = \gamma(s,(\Wedge^2\circ\rhostd\circ\varphi_2)\otimes\tau,\psi),
\]
for all irreducible~$r$-dimensional representations of~$W_F$ with~$1\le r\le 2$. Moreover, since~$\Wedge^2\circ \rhostd \simeq\rhostd\oplus\Ad$ is the sum of the two fundamental representations of~$\G_2(\mathbb{C})$, multiplicativity of~$\gamma$-factors implies that
\[
\gamma(s,(\varrho\circ\varphi_1)\otimes\tau,\psi) = \gamma(s,(\varrho\circ\varphi_2)\otimes\tau,\psi),
\]
for all fundamental representations~$\varrho$ of~$\G_2(\mathbb{C})$ and all irreducible~$r$-dimensional representations of~$W_F$ with~$1\le r\le 2$. 

Thus the necessity of twisting by~$3$-dimensional representations of~$W_F$ to distinguish parameters is not an artefact of only considering the standard representation of~$\G_2(\mathbb{C})$. 
%
\end{rmk}

\subsection{The optimal local converse theorem for~$\G_2$}\label{g2lct}
Finally, we interpret the results of the previous sections in terms of irreducible representations of~$\G_2(F)$. We have, from~\cite{GS23a} (see also~\cite{HKT21} and~\cite{AX22}) a local Langlands correspondence for~$\G_2(F)$; for~$\pi$ an irreducible representation of~$\G_2(F)$, we write~$\varphi_\pi: \WD_F\to \G_2(\mathbb{C})$ for the corresponding Langlands parameter, and by~$\Pi_{\pi}$ the~$L$-packet of~$\pi$, which contains a unique generic representation. For~$\tau$ an irreducible representation of~$\GL_m(F)$ with Langlands parameter~$\varphi_\tau$, as in~\cite{GS23b} we \emph{define}
\[
\gamma(s,\pi\rtimes\tau,\varrho,\psi) := \gamma(s,(\varrho\circ\varphi_{\pi})\otimes\varphi_\tau,\psi),
\]
for~$\varrho$ any algebraic representation of~$\G_2(\BC)$.

Writing~$\Irrgen_\scusp(\G(F))$ for the set of equivalence classes of irreducible generic supercuspidal representations of the points~$\G(F)$ of a reductive group, the map in~\eqref{eqn:PGLtoG} corresponds to a map
\[
\Irrgen_{\scusp}(\PGL_3) \to \Irrgen_{\scusp}(\G_2),
\]
which is given by an exceptional theta correspondence, and is in fact part of how the Langlands correspondence is defined (see also~\cite[Theorem~1.8]{SWe11}). Since the parameters used in the construction in the proof of Theorem~\ref{G2} are indeed for supercuspidal representations of~$\PGL_3(F)$, the previous results of this section (also using Remark~\ref{rem:G2}) translate to give:

\begin{thm}\label{thm:converseG2}
\begin{enumerate}
\item Let~$\pi_1,\pi_2$ be irreducible generic representations of~$\G_2(F)$ which are~$\gamma$-equivalent to level~$3$. Then~$\pi_1\simeq\pi_2$.
\item Suppose~$p>3$. Then there are inequivalent generic supercuspidal representations~$\pi_1,\pi_2$ of~$\G_2(F)$ such that
\[
\gamma(s,\pi_1\rtimes\tau,\varrho,\psi) := \gamma(s,\pi_2\rtimes\tau,\varrho,\psi) ,
\]
for~$\varrho$ either fundamental representation of~$\G_2(\BC)$, and all supercuspidal representations of~$\GL_r(F)$, with~$1\le r\le 2$; in particular, they are~$\gamma$-equivalent to level~$2$.
\end{enumerate}
\end{thm}

\section{Ramakrishnan conjecture}\label{sec:ramakrishnan}
In this final section, we show that we also cannot use only~$\gamma$-functions of character twists to determine a supercuspidal representation of~$\GL_N(F)$. We recall that, in~\cite[Conjecture~1.2]{YZ22}, it is conjectured that if~$\pi_1,\pi_2$ are unitarizable supercuspidal representations of~$\GL_N(F)$ such that
\[
\gamma(s, \pi_1\rtimes \eta,\Wedge^i, \psi) = \gamma(s, \pi_2\rtimes \eta,\Wedge^i, \psi),
\]
for~$i=1,\ldots,\left\lfloor\tfrac N2\right\rfloor$ and~$\eta$ any character of~$F^\times$, then~$\pi_1\cong\pi_2$. This is of course true for~$N=2,3$ since it is then simply the Local Converse Theorem for~$\GL_N(F)$. On the other hand, for~$N\ge 6$ the requisite~$\gamma$-functions have not yet been defined on the automorphic side of the Langlands correspondence so we interpret these~$\gamma$-factors on the Galois side as for~$\G_2$ and as in the introduction. 

We show here that this conjecture is false already for~$\GL_4(F)$, when~$p$ is odd. Indeed, we give examples of inequivalent irreducible~$4$-dimensional representations~$\varphi_1,\varphi_2$ of~$W_F$, corresponding to irreducible unitarizable supercuspidal representations of $\GL_4(F)$, such that
\begin{equation}\label{eqn:ram}
\gamma(s, (\Wedge^i\circ\varphi_1)\otimes\eta, \psi) = \gamma(s, (\Wedge^i\circ\varphi_2)\otimes\eta, \psi),
\end{equation}
for~$i=1,\ldots,4$ and all characters~$\eta$ of~$W_F$. 

First, we remark that if a supercuspidal representation of $\GL_N(F)$ has unitary central character, then it is unitarizable (see \cite[p.4]{KT24}).  We also remark that if $(E/F, \chi)$ is an admissible pair, then its corresponding supercuspidal representation $\pi_{\chi}$ via LLC has unitary central character if and only if $\chi|_{F^{\times}}$ is unitary.  Note that $\chi|_{\mathcal{O}_F^{\times}}$ is automatically unitary since $\mathcal{O}_F^{\times}$ is compact.  

So suppose~$p$ is odd and let~$(E/F,\chi)$ be any totally ramified unitary admissible pair of degree~$4$ whose depth~$d$ has least denominator~$2$. For example, we could choose~$\beta$ to be the quasi-minimal element~$\varpi_E^{-10}+\varpi_E^{-7}$, where~$\varpi_E$ is a uniformizer of~$E$, and take~$\chi$ to be any unitary quasi-character of~$E^\times$ represented by~$\beta$. Then we take~$(E/F,\chi')$ to be the admissible pair such that~$\chi^{-1}\chi'$ is unramified of order~$2$. Then~$\chi,\chi'$ coincide on~$F^\times$, and also on~$\beta$ since~$4\val(\beta)$ is an even integer.

Now note that
\[
\rho_\chi \cong \rho_{\chi'}\otimes \eta_8,
\]
where~$\eta_n$ denotes an unramified character of~$W_F$ of order~$n$ (the choice is unimportant), since~$\eta_8\circ\N_{E/F}$ is the unramified quadratic character of~$E^\times$. In particular, since also~$\rho_\chi\cong\rho_\chi\otimes\eta_4$, we have~$\Wedge^2\circ\rho_\chi\cong\Wedge^2\circ\rho_{\chi'}$, and similarly,~$\omega_\chi=\Wedge^4\circ\rho_\chi=\Wedge^4\circ\rho_\chi=\omega_{\chi'}$. On the other hand, the pairing~$\Wedge^3\circ\rho_\chi\times\Wedge^1\circ\rho_\chi\to\Wedge^4\circ\rho_\chi$ gives us~$\Wedge^3\circ\rho_\chi\cong\rho_\chi^\vee \otimes\omega_{\chi}$.

Thus, putting~$\varphi_1=\rho_\chi$ and~$\varphi_2=\rho_{\chi'}$, to prove~\eqref{eqn:ram} we are reduced to verifying
\[
\gamma(s, \rho_{\chi}\otimes\eta, \psi) = \gamma(s, \rho_{\chi'}\otimes\eta, \psi) \quad\text{ and }\quad \gamma(s, \rho_{\chi^{-1}}\otimes\eta, \psi) = \gamma(s, \rho_{\chi'^{-1}}\otimes\eta, \psi),
\]
for all characters~$\eta$ of~$W_F$. But these follow immediately from Proposition~\ref{prop:basic2}.

\color{black}

\bibliographystyle{alpha}
\bibliography{sharpness}

@article{AX22,
      author={Anne-Marie Aubert and Yujie Xu},
      title={{T}he {E}xplicit {L}ocal {L}anglands {C}orrespondence for {$G_2$}}, 
      year={2023},
      eprint={2208.12391},
      archivePrefix={arXiv},
      primaryClass={math.RT},
      url={https://arxiv.org/abs/2208.12391}, 
      NOTE={arXiv:2208.12391},
}

@article {M24,
    AUTHOR = {Matringe, N.},
     TITLE = {Local converse theorems and {L}anglands parameters},
     YEAR = {2024},
     eprint = {2409.20240},
      archivePrefix = {arXiv},
      primaryClass = {math.RT},
      URL = {https://arxiv.org/pdf/2409.20240},
      NOTE={arXiv:2409.20240},

}

@article {CKPSS,
    AUTHOR = {Cogdell, J. W. and Kim, H. H. and Piatetski-Shapiro, I. I. and
              Shahidi, F.},
     TITLE = {Functoriality for the classical groups},
   JOURNAL = {Publ. Math. Inst. Hautes \'Etudes Sci.},
  FJOURNAL = {Publications Math\'ematiques. Institut de Hautes \'Etudes
              Scientifiques},
    NUMBER = {99},
      YEAR = {2004},
     PAGES = {163--233},
      ISSN = {0073-8301,1618-1913},
   MRCLASS = {22E46 (22E55)},
  MRNUMBER = {2075885},
       DOI = {10.1007/s10240-004-0020-z},
       URL = {https://doi.org/10.1007/s10240-004-0020-z},
}

@book {BF83,
    AUTHOR = {Bushnell, Colin J. and Fr\"ohlich, Albrecht},
     TITLE = {Gauss sums and {$p$}-adic division algebras},
    SERIES = {Lecture Notes in Mathematics},
    VOLUME = {987},
 PUBLISHER = {Springer-Verlag, Berlin-New York},
      YEAR = {1983},
     PAGES = {xi+187},
      ISBN = {3-540-12290-7},
   MRCLASS = {12B27 (12B37 22E50)},
  MRNUMBER = {701540},
MRREVIEWER = {Ernst-Wilhelm\ Zink},
}

@misc {A22,
    AUTHOR = {Adrian, Moshe},
     TITLE = {On the sharpness of the bound for the local converse theorem of {$p$}-adic {$\mathrm{GL}_N$}, general {$N$}},
     YEAR = {2023},
     eprint = {2303.08656},
      archivePrefix = {arXiv},
      primaryClass = {math.RT},
      URL = {https://arxiv.org/abs/2303.08656},
      NOTE={arXiv:2303.08656},
 }

@article {AHKO,
    AUTHOR = {Adrian, M. and Henniart, G. and Kaplan, E. and
              Oi, M},
     TITLE = {Simple supercuspidal L-packets of split special orthogonal groups over dyadic fields},
   JOURNAL = { J. Lond. Math. Soc.},
  FJOURNAL = {Journal of the London Mathematical Society},
    VOLUME = {112},
      YEAR = {2025},
     PAGES = {1--62},
      ISSN = {0024-6107},
   MRCLASS = {11S37 (11F70 11F80 11F85 22E50)},
  MRNUMBER = {MR4927215},
MRREVIEWER = {},
       DOI = {10.1112/jlms.70223},
       URL = {https://doi-org.queens.ezproxy.cuny.edu/10.1112/jlms.70223},
}

@article {ALST18,
    AUTHOR = {Adrian, Moshe and Liu, Baiying and Stevens, Shaun and Tam, Geo
              Kam-Fai},
     TITLE = {On the sharpness of the bound for the local converse theorem
              of {$p$}-adic {$\rm GL_{prime}$}},
   JOURNAL = {Proc. Amer. Math. Soc. Ser. B},
  FJOURNAL = {Proceedings of the American Mathematical Society. Series B},
    VOLUME = {5},
      YEAR = {2018},
     PAGES = {6--17},
      ISSN = {2330-1511},
   MRCLASS = {11S70 (11F85 22E50 22E55)},
  MRNUMBER = {3765745},
MRREVIEWER = {Wemedh\ Aeal},
       DOI = {10.1090/bproc/32},
       URL = {https://doi.org/10.1090/bproc/32},
}

@book {BK93,
    AUTHOR = {Bushnell, Colin J. and Kutzko, Philip C.},
     TITLE = {The admissible dual of {${\mathrm{GL}}(N)$} via compact open
              subgroups},
    SERIES = {Annals of Mathematics Studies},
    VOLUME = {129},
 PUBLISHER = {Princeton University Press, Princeton, NJ},
      YEAR = {1993},
     PAGES = {xii+313},
      ISBN = {0-691-03256-4; 0-691-02114-7},
   MRCLASS = {22E50 (22-02)},
  MRNUMBER = {1204652},
MRREVIEWER = {Mark\ Reeder},
       DOI = {10.1515/9781400882496},
       URL = {https://doi.org/10.1515/9781400882496},
}

@book {BH06,
    AUTHOR = {Bushnell, Colin J. and Henniart, Guy},
     TITLE = {The local {L}anglands conjecture for {$\mathrm{GL}(2)$}},
    SERIES = {Grundlehren der mathematischen Wissenschaften [Fundamental
              Principles of Mathematical Sciences]},
    VOLUME = {335},
 PUBLISHER = {Springer-Verlag, Berlin},
      YEAR = {2006},
     PAGES = {xii+347},
      ISBN = {978-3-540-31486-8; 3-540-31486-5},
   MRCLASS = {22E50 (11-02 11S37 22-02)},
  MRNUMBER = {2234120},
MRREVIEWER = {Alexandru\ Ioan\ Badulescu},
       DOI = {10.1007/3-540-31511-X},
       URL = {https://doi.org/10.1007/3-540-31511-X},
}

@article {BH14,
    AUTHOR = {Bushnell, Colin J. and Henniart, Guy},
     TITLE = {Langlands parameters for epipelagic representations of {${\mathrm{GL}}_n$}},
   JOURNAL = {Math. Ann.},
  FJOURNAL = {Mathematische Annalen},
    VOLUME = {358},
      YEAR = {2014},
    NUMBER = {1-2},
     PAGES = {433--463},
      ISSN = {0025-5831,1432-1807},
   MRCLASS = {22E50 (11S37)},
  MRNUMBER = {3158004},
MRREVIEWER = {Anton\ Deitmar},
       DOI = {10.1007/s00208-013-0962-x},
       URL = {https://doi.org/10.1007/s00208-013-0962-x},
}

@article {Ch19,
    AUTHOR = {Chai, Jingsong},
     TITLE = {Bessel functions and local converse conjecture of {J}acquet},
   JOURNAL = {J. Eur. Math. Soc. (JEMS)},
  FJOURNAL = {Journal of the European Mathematical Society (JEMS)},
    VOLUME = {21},
      YEAR = {2019},
    NUMBER = {6},
     PAGES = {1703--1728},
      ISSN = {1435-9855,1435-9863},
   MRCLASS = {11F70 (22E50)},
  MRNUMBER = {3945739},
MRREVIEWER = {Ivan\ Mati\'c},
       DOI = {10.4171/JEMS/870},
       URL = {https://doi.org/10.4171/JEMS/870},
}

@article {CG23,
    AUTHOR = {Chenevier, G. and Gan, W.T.},
     TITLE = {{$\mathrm{Spin}(7)$} is unacceptable},
   JOURNAL = {Peking Math J.},
  FJOURNAL = {JPeking Math Journal]},
      YEAR = {2024},
       DOI = {10.1007/s42543-023-00083-3},
       URL = {https://doi.org/10.1007/s42543-023-00083-3},
}

@article {G95,
    AUTHOR = {Griess, Jr., Robert L.},
     TITLE = {Basic conjugacy theorems for {$G_2$}},
   JOURNAL = {Invent. Math.},
  FJOURNAL = {Inventiones Mathematicae},
    VOLUME = {121},
      YEAR = {1995},
    NUMBER = {2},
     PAGES = {257--277},
      ISSN = {0020-9910,1432-1297},
   MRCLASS = {20G20 (17A35 20D06 22E10)},
  MRNUMBER = {1346206},
MRREVIEWER = {James\ E.\ Humphreys},
       DOI = {10.1007/BF01884298},
       URL = {https://doi.org/10.1007/BF01884298},
 }

@article {HKT21,
    AUTHOR = {Harris, Michael and Khare, Chandrashekhar B. and Thorne, Jack A.},
     TITLE = {A local Langlands parameterization for generic supercuspidal representations of {$p$}-adic {$G_2$}.},
   JOURNAL = { Ann. Sci. Éc. Norm. Supér. },
  FJOURNAL = {Annales scientifiques de l'Ecole normale supérieure},
    VOLUME = {56},
      YEAR = {2023},
     PAGES = {257--286},
      ISSN = {0012-9593},
   MRCLASS = {11R39 (11F80 20G41 22E50)},
  MRNUMBER = {MR4637132},
MRREVIEWER = {},
       DOI = {10.24033/asens.2533},
       URL = {https://doi.org/10.24033/asens.2230},      
 }

@article {HZ23,
    AUTHOR = {Hazeltine, Alexander and Liu, Baiying},
     TITLE = {On the local converse Theorem for split {$\mathrm{SO}_{2l}$}},
   JOURNAL = { Represent. Theory },
  FJOURNAL = {Representation Theory},
    VOLUME = {29},
      YEAR = {2025},
     PAGES = {209--255},
      ISSN = {1088-4165},
   MRCLASS = { 11F70 (11F85 22E50)},
  MRNUMBER = {MR4880447},
MRREVIEWER = {},
       DOI = {10.1090/ert/687},
       URL = {https://doi-org.queens.ezproxy.cuny.edu/10.1090/ert/687},     
}

@article {Ho77,
    AUTHOR = {Howe, Roger E.},
     TITLE = {Tamely ramified supercuspidal representations of {${\mathrm{GL}}_{n}$}},
   JOURNAL = {Pacific J. Math.},
  FJOURNAL = {Pacific Journal of Mathematics},
    VOLUME = {73},
      YEAR = {1977},
    NUMBER = {2},
     PAGES = {437--460},
      ISSN = {0030-8730,1945-5844},
   MRCLASS = {22E50},
  MRNUMBER = {492087},
MRREVIEWER = {Allan\ J.\ Silberger},
       URL = {http://projecteuclid.org/euclid.pjm/1102810618},
}

@article {JL18,
    AUTHOR = {Jacquet, Herv\'e{} and Liu, Baiying},
     TITLE = {On the local converse theorem for {$p$}-adic {${\mathrm{GL}}_n$}},
   JOURNAL = {Amer. J. Math.},
  FJOURNAL = {American Journal of Mathematics},
    VOLUME = {140},
      YEAR = {2018},
    NUMBER = {5},
     PAGES = {1399--1422},
      ISSN = {0002-9327,1080-6377},
   MRCLASS = {11S70 (11F85 22E57)},
  MRNUMBER = {3862069},
MRREVIEWER = {Ivan\ Mati\'c},
       DOI = {10.1353/ajm.2018.0035},
       URL = {https://doi.org/10.1353/ajm.2018.0035},
}

@article {JPSS83,
    AUTHOR = {Jacquet, H. and Piatetskii-Shapiro, I. I. and Shalika, J. A.},
     TITLE = {Rankin-{S}elberg convolutions},
   JOURNAL = {Amer. J. Math.},
  FJOURNAL = {American Journal of Mathematics},
    VOLUME = {105},
      YEAR = {1983},
    NUMBER = {2},
     PAGES = {367--464},
      ISSN = {0002-9327,1080-6377},
   MRCLASS = {11F67 (11F70 11R39 22E55)},
  MRNUMBER = {701565},
MRREVIEWER = {Freydoon\ Shahidi},
       DOI = {10.2307/2374264},
       URL = {https://doi.org/10.2307/2374264},
}

@article {JNS15,
    AUTHOR = {Jiang, Dihua and Nien, Chufeng and Stevens, Shaun},
     TITLE = {Towards the {J}acquet conjecture on the local converse problem
              for {$p$}-adic {${\mathrm{GL}}_n$}},
   JOURNAL = {J. Eur. Math. Soc. (JEMS)},
  FJOURNAL = {Journal of the European Mathematical Society (JEMS)},
    VOLUME = {17},
      YEAR = {2015},
    NUMBER = {4},
     PAGES = {991--1007},
      ISSN = {1435-9855,1435-9863},
   MRCLASS = {22E50 (11F70 11F85)},
  MRNUMBER = {3349305},
MRREVIEWER = {Ivan\ Mati\'c},
       DOI = {10.4171/JEMS/524},
       URL = {https://doi.org/10.4171/JEMS/524},
}

@article {JS03,
    AUTHOR = {Jiang, Dihua and Soudry, David},
     TITLE = {The local converse theorem for {${\mathrm{SO}}(2n+1)$} and
              applications},
   JOURNAL = {Ann. of Math. (2)},
  FJOURNAL = {Annals of Mathematics. Second Series},
    VOLUME = {157},
      YEAR = {2003},
    NUMBER = {3},
     PAGES = {743--806},
      ISSN = {0003-486X,1939-8980},
   MRCLASS = {11S37 (22E50)},
  MRNUMBER = {1983781},
       DOI = {10.4007/annals.2003.157.743},
       URL = {https://doi.org/10.4007/annals.2003.157.743},
}

@article {Ku80,
    AUTHOR = {Kutzko, Philip},
     TITLE = {The {L}anglands conjecture for {${\mathrm{GL}}_{2}$} of a local
              field},
   JOURNAL = {Ann. of Math. (2)},
  FJOURNAL = {Annals of Mathematics. Second Series},
    VOLUME = {112},
      YEAR = {1980},
    NUMBER = {2},
     PAGES = {381--412},
      ISSN = {0003-486X},
   MRCLASS = {12B27 (10D40 22E50)},
  MRNUMBER = {592296},
MRREVIEWER = {Freydoon\ Shahidi},
       DOI = {10.2307/1971151},
       URL = {https://doi.org/10.2307/1971151},
}

@article {L94,
    AUTHOR = {Larsen, Michael},
     TITLE = {On the conjugacy of element-conjugate homomorphisms},
   JOURNAL = {Israel J. Math.},
  FJOURNAL = {Israel Journal of Mathematics},
    VOLUME = {88},
      YEAR = {1994},
    NUMBER = {1-3},
     PAGES = {253--277},
      ISSN = {0021-2172,1565-8511},
   MRCLASS = {20G20 (20C15 22E15)},
  MRNUMBER = {1303498},
MRREVIEWER = {James\ E.\ Humphreys},
       DOI = {10.1007/BF02937514},
       URL = {https://doi.org/10.1007/BF02937514},
}

@article {M84,
    AUTHOR = {Moy, Allen},
     TITLE = {The irreducible orthogonal and symplectic {G}alois
              representations of a {$p$}-adic field (the tame case)},
   JOURNAL = {J. Number Theory},
  FJOURNAL = {Journal of Number Theory},
    VOLUME = {19},
      YEAR = {1984},
    NUMBER = {3},
     PAGES = {341--344},
      ISSN = {0022-314X,1096-1658},
   MRCLASS = {11S20 (11S37)},
  MRNUMBER = {769787},
MRREVIEWER = {Hans\ Opolka},
       DOI = {10.1016/0022-314X(84)90076-3},
       URL = {https://doi.org/10.1016/0022-314X(84)90076-3},
}

@article {M86,
    AUTHOR = {Moy, Allen},
     TITLE = {Local constants and the tame {L}anglands correspondence},
   JOURNAL = {Amer. J. Math.},
  FJOURNAL = {American Journal of Mathematics},
    VOLUME = {108},
      YEAR = {1986},
    NUMBER = {4},
     PAGES = {863--930},
      ISSN = {0002-9327,1080-6377},
   MRCLASS = {11S37 (22E50)},
  MRNUMBER = {853218},
MRREVIEWER = {David\ Manderscheid},
       DOI = {10.2307/2374518},
       URL = {https://doi.org/10.2307/2374518},
}

@article {SWe11,
    AUTHOR = {Savin, Gordan and Weissman, Martin H.},
     TITLE = {Dichotomy for generic supercuspidal representations of
              {$G_2$}},
   JOURNAL = {Compos. Math.},
  FJOURNAL = {Compositio Mathematica},
    VOLUME = {147},
      YEAR = {2011},
    NUMBER = {3},
     PAGES = {735--783},
      ISSN = {0010-437X,1570-5846},
   MRCLASS = {22E50 (11F70 11R39)},
  MRNUMBER = {2801399},
MRREVIEWER = {Luis\ Alberto\ Lomel\'i},
       DOI = {10.1112/S0010437X10005178},
       URL = {https://doi.org/10.1112/S0010437X10005178},
}

@article {GS23a,
    AUTHOR = {Gan, Wee Teck and Savin, Gordan},
     TITLE = {The local {L}anglands conjecture for {$G_2$}},
   JOURNAL = {Forum Math. Pi},
  FJOURNAL = {Forum of Mathematics. Pi},
    VOLUME = {11},
      YEAR = {2023},
     PAGES = {Paper No. e28, 42},
      ISSN = {2050-5086},
   MRCLASS = {11S37 (11F27 11F70 22E50)},
  MRNUMBER = {4658199},
       DOI = {10.1017/fmp.2023.27},
       URL = {https://doi.org/10.1017/fmp.2023.27},
}

@misc{GS23b,
      title={A theory of {$\gamma$}-factors for {$\mathrm{G}_2 \times \mathrm{GL}_r$}}, 
      author={Wee Teck Gan and Gordan Savin},
      year={2023},
      eprint={2308.12561},
      archivePrefix={arXiv},
      primaryClass={math.NT},
      url={https://arxiv.org/abs/2308.12561}, 
      NOTE={arXiv:2308.12561},
}

@article {S84,
    AUTHOR = {Shahidi, Freydoon},
     TITLE = {Fourier transforms of intertwining operators and {P}lancherel
              measures for {${\mathrm{GL}}(n)$}},
   JOURNAL = {Amer. J. Math.},
  FJOURNAL = {American Journal of Mathematics},
    VOLUME = {106},
      YEAR = {1984},
    NUMBER = {1},
     PAGES = {67--111},
      ISSN = {0002-9327,1080-6377},
   MRCLASS = {22E50 (11S37)},
  MRNUMBER = {729755},
MRREVIEWER = {Allan\ J.\ Silberger},
       DOI = {10.2307/2374430},
       URL = {https://doi.org/10.2307/2374430},
}

@incollection {T50,
    AUTHOR = {Tate, J. T.},
     TITLE = {Fourier analysis in number fields, and {H}ecke's
              zeta-functions},
 BOOKTITLE = {Algebraic {N}umber {T}heory ({P}roc. {I}nstructional {C}onf.,
              {B}righton, 1965)},
     PAGES = {305--347},
 PUBLISHER = {Academic Press, London},
      YEAR = {1967},
   MRCLASS = {10.41},
  MRNUMBER = {217026},
 }

@article {YZ24,
    AUTHOR = {Yan, Pan and Zhang, Qing},
     TITLE = {On a refined local converse theorem for {$\mathrm{SO}(4)$}},
   JOURNAL = { Proc. Amer. Math. Soc. },
  FJOURNAL = {Proceedings of the American Mathematical Society},
    VOLUME = {152},
      YEAR = {2024},
     PAGES = {4959--4979},
      ISSN = {0002-9939},
   MRCLASS = { 22E50 (11F70)},
  MRNUMBER = {MR4802646},
MRREVIEWER = {Xu, Bin},
       DOI = {10.1090/proc/16945},
       URL = {https://doi-org.queens.ezproxy.cuny.edu/10.1090/proc/16945}, 
}

@article {Z18,
    AUTHOR = {Zhang, Qing},
     TITLE = {A local converse theorem for {${\mathrm{Sp}}_{2r}$}},
   JOURNAL = {Math. Ann.},
  FJOURNAL = {Mathematische Annalen},
    VOLUME = {372},
      YEAR = {2018},
    NUMBER = {1-2},
     PAGES = {451--488},
      ISSN = {0025-5831,1432-1807},
   MRCLASS = {11F70 (22E50)},
  MRNUMBER = {3856818},
MRREVIEWER = {Arnab\ Mitra},
       DOI = {10.1007/s00208-017-1623-2},
       URL = {https://doi.org/10.1007/s00208-017-1623-2},
}

@article {KSS21,
    AUTHOR = {Kurinczuk, Robert and Skodlerack, Daniel and Stevens, Shaun},
     TITLE = {Endo-parameters for {$p$}-adic classical groups},
   JOURNAL = {Invent. Math.},
  FJOURNAL = {Inventiones Mathematicae},
    VOLUME = {223},
      YEAR = {2021},
    NUMBER = {2},
     PAGES = {597--723},
      ISSN = {0020-9910,1432-1297},
   MRCLASS = {22E50 (11F70)},
  MRNUMBER = {4209861},
MRREVIEWER = {Neven\ Grbac},
       DOI = {10.1007/s00222-020-00997-0},
       URL = {https://doi.org/10.1007/s00222-020-00997-0},
}

@article {YZ22,
    AUTHOR = {Ye, Rongqing and Zelingher, Elad},
     TITLE = {Exterior square gamma factors for cuspidal representations of
              {${\rm GL}_n$}: simple supercuspidal representations},
   JOURNAL = {Ramanujan J.},
  FJOURNAL = {Ramanujan Journal. An International Journal Devoted to the
              Areas of Mathematics Influenced by Ramanujan},
    VOLUME = {58},
      YEAR = {2022},
    NUMBER = {4},
     PAGES = {1043--1074},
      ISSN = {1382-4090,1572-9303},
   MRCLASS = {22E50 (11F66)},
  MRNUMBER = {4451511},
MRREVIEWER = {Subhajit\ Jana},
       DOI = {10.1007/s11139-021-00476-x},
       URL = {https://doi.org/10.1007/s11139-021-00476-x},
}

@incollection {Ram94,
    AUTHOR = {Ramakrishnan, Dinakar},
     TITLE = {Pure motives and automorphic forms},
 BOOKTITLE = {Motives ({S}eattle, {WA}, 1991)},
    SERIES = {Proc. Sympos. Pure Math.},
    VOLUME = {55, Part 2},
     PAGES = {411--446},
 PUBLISHER = {Amer. Math. Soc., Providence, RI},
      YEAR = {1994},
      ISBN = {0-8218-1637-3},
   MRCLASS = {11R39 (11F80 11G09)},
  MRNUMBER = {1265561},
MRREVIEWER = {Richard\ Taylor},
       DOI = {10.1090/pspum/055.2/1265561},
       URL = {https://doi.org/10.1090/pspum/055.2/1265561},
}

@book {FH,
    AUTHOR = {Fulton, William and Harris, Joe},
     TITLE = {Representation theory},
    SERIES = {Graduate Texts in Mathematics},
    VOLUME = {129},
      NOTE = {A first course,
              Readings in Mathematics},
 PUBLISHER = {Springer-Verlag, New York},
      YEAR = {1991},
     PAGES = {xvi+551},
      ISBN = {0-387-97527-6; 0-387-97495-4},
   MRCLASS = {20G05 (17B10 20G20 22E46)},
  MRNUMBER = {1153249},
MRREVIEWER = {James\ E.\ Humphreys},
       DOI = {10.1007/978-1-4612-0979-9},
       URL = {https://doi.org/10.1007/978-1-4612-0979-9},
}

@misc {KT24,
    AUTHOR = {Kaletha, Tasho and Ta{\"i}bi, Olivier},
     TITLE = {The local {L}anglands conjecture},
     YEAR = {},
      URL = {https://otaibi.perso.math.cnrs.fr/kaletha-taibi-llc.pdf},
      NOTE={https://otaibi.perso.math.cnrs.fr/kaletha-taibi-llc.pdf},
}

@article {NZ21,
    AUTHOR = {Nien, Chufeng and Zhang, Lei},
     TITLE = {Converse theorem of {G}auss sums},
      NOTE = {With an appendix by Zhiwei Yun},
   JOURNAL = {J. Number Theory},
  FJOURNAL = {Journal of Number Theory},
    VOLUME = {221},
      YEAR = {2021},
     PAGES = {365--388},
      ISSN = {0022-314X,1096-1658},
   MRCLASS = {11T24},
  MRNUMBER = {4203574},
MRREVIEWER = {Arne\ Winterhof},
       DOI = {10.1016/j.jnt.2020.10.022},
       URL = {https://doi.org/10.1016/j.jnt.2020.10.022},
}

@article {Blo04,
    AUTHOR = {Blondel, Corinne},
     TITLE = {{${\rm SP}(2N)$}-covers for self-contragredient supercuspidal
              representations of {${\rm GL}(N)$}},
   JOURNAL = {Ann. Sci. \'Ecole Norm. Sup. (4)},
  FJOURNAL = {Annales Scientifiques de l'\'Ecole Normale Sup\'erieure.
              Quatri\`eme S\'erie},
    VOLUME = {37},
      YEAR = {2004},
    NUMBER = {4},
     PAGES = {533--558},
      ISSN = {0012-9593},
   MRCLASS = {22E50},
  MRNUMBER = {2097892},
MRREVIEWER = {D.\ Mili\v ci\'c},
       DOI = {10.1016/j.ansens.2003.10.003},
       URL = {https://doi.org/10.1016/j.ansens.2003.10.003},
}

@article {V17,
    AUTHOR = {Varma, Sandeep},
     TITLE = {On descent and the generic packet conjecture. },
   JOURNAL = {Forum Math},
  FJOURNAL = {Forum Mathematicum},
    VOLUME = {29},
      YEAR = {2017},
     PAGES = {111--155},
      ISSN = {2330-1511},
   MRCLASS = {11S70 (11F85 22E50 22E55)},
  MRNUMBER = {3765745},
MRREVIEWER = {Wemedh\ Aeal},
       DOI = {10.1090/bproc/32},
       URL = {https://doi.org/10.1090/bproc/32},
}

@article {GI16,
    AUTHOR = {Gan, Wee Teck and Ichino, Atsushi},
     TITLE = {The {G}ross-{P}rasad conjecture and local theta
              correspondence},
   JOURNAL = {Invent. Math.},
  FJOURNAL = {Inventiones Mathematicae},
    VOLUME = {206},
      YEAR = {2016},
    NUMBER = {3},
     PAGES = {705--799},
      ISSN = {0020-9910,1432-1297},
   MRCLASS = {11F70 (11F85 22E50)},
  MRNUMBER = {3573972},
MRREVIEWER = {Ivan\ Mati\'c},
       DOI = {10.1007/s00222-016-0662-8},
       URL = {https://doi.org/10.1007/s00222-016-0662-8},
}

@misc{Jo22,
      AUTHOR={Jo, Yeongseong},
      TITLE={The local converse theorem for odd special orthogonal and symplectic groups in positive characteristic}, 
      YEAR={2025},
      eprint={2205.09004},
      archivePrefix={arXiv},
      primaryClass={math.RT},
      URL={https://arxiv.org/abs/2205.09004}, 
      NOTE={arXiv:2205.09004},
}

@misc{JL24,
     AUTHOR={Jantzen, Chris and Liu, Baiying},
      TITLE={The generic dual of {$p$}-adic groups and applications}, 
      YEAR={2024},
       eprint={2404.07111},
      archivePrefix={arXiv},
      primaryClass={math.RT},
      URL={https://arxiv.org/abs/2404.07111}, 
      NOTE={arXiv:2404.07111},
}

@book {A15,
    AUTHOR = {Arthur, James},
     TITLE = {The endoscopic classification of representations},
    SERIES = {American Mathematical Society Colloquium Publications},
    VOLUME = {61},
      NOTE = {Orthogonal and symplectic groups},
 PUBLISHER = {American Mathematical Society, Providence, RI},
      YEAR = {2013},
     PAGES = {xviii+590},
      ISBN = {978-0-8218-4990-3},
   MRCLASS = {22E55 (11F66 11F70 11F72 11R37 20G25 22E50)},
  MRNUMBER = {3135650},
MRREVIEWER = {Dihua\ Jiang},
       DOI = {10.1090/coll/061},
       URL = {https://doi.org/10.1090/coll/061},
}

@misc{LS25,
      AUTHOR={David C. Luo and Shaun Stevens},
      TITLE={On the {L}ocal {C}onverse {T}heorem for depth {$\frac{1}{N}$} supercuspidal representations of {$\text{GL}(2N, F)$}}, 
      year={2025},
      eprint={2505.22357},
      archivePrefix={arXiv},
      primaryClass={math.RT},
      URL={https://arxiv.org/abs/2505.22357}, 
      NOTE={arXiv:2505.22357},
}

@article {Yu22,
    AUTHOR = {Yu, Jun},
     TITLE = {Acceptable compact {L}ie groups},
   JOURNAL = {Peking Math. J.},
  FJOURNAL = {Peking Mathematical Journal},
    VOLUME = {5},
      YEAR = {2022},
    NUMBER = {2},
     PAGES = {427--446},
      ISSN = {2096-6075,2524-7182},
   MRCLASS = {22E15 (20G20)},
  MRNUMBER = {4492659},
MRREVIEWER = {Alexandre\ Jos\'e\ Santana},
       DOI = {10.1007/s42543-021-00038-6},
       URL = {https://doi.org/10.1007/s42543-021-00038-6},
}

@article {CDFZ25,
    AUTHOR = {Cunningham, Clifton and Dijols, Sarah and Fiori, Andrew and
              Zhang, Qing},
     TITLE = {Whittaker normalization of {$p$}-adic {ABV}-packets and
              {V}ogan's conjecture for tempered representations},
   JOURNAL = {Int. Math. Res. Not. IMRN},
  FJOURNAL = {International Mathematics Research Notices. IMRN},
      YEAR = {2025},
    NUMBER = {20},
     PAGES = {Paper No. rnaf316, 26},
      ISSN = {1073-7928,1687-0247},
   MRCLASS = {22E50},
  MRNUMBER = {4976047},
MRREVIEWER = {Ivan\ Mati\'c},
       DOI = {10.1093/imrn/rnaf316},
       URL = {https://doi.org/10.1093/imrn/rnaf316},
}

\end{document}